\documentclass[journal]{IEEEtran}
\usepackage{amsmath,amssymb,amsthm,url,wrapfig}
\usepackage{graphicx,subfigure}

\newcommand{\bracket}[1]{\ensuremath{\left[ #1 \right]}}
\newcommand{\braces}[1]{\ensuremath{\left\{ #1 \right\}}}
\newcommand{\parenth}[1]{\ensuremath{\left( #1 \right)}}

\newcommand{\refeqn}[1]{(\ref{eqn:#1})}

\newcommand{\tr}[1]{\mbox{tr}\ensuremath{\negthickspace\bracket{#1}}}
\newcommand{\trs}[1]{\mbox{tr}\ensuremath{\!\bracket{#1}}}

\newcommand{\SO}{\ensuremath{\mathsf{SO(3)}}}
\newcommand{\T}{\ensuremath{\mathsf{T}}}

\newcommand{\so}{\ensuremath{\mathfrak{so}(3)}}

\renewcommand{\Re}{\ensuremath{\mathbb{R}}}
\newcommand{\Sph}{\ensuremath{\mathsf{S}}}

\newcommand{\RNum}[1]{\uppercase\expandafter{\romannumeral #1\relax}}
\newcommand{\RI}{\text{\RNum{1}}}
\newcommand{\RII}{\text{\RNum{2}}}
\newcommand{\RIII}{\text{\RNum{3}}}

\DeclareMathOperator*{\argmin}{arg\,min}

\title{\LARGE \bf
Global Exponential Attitude Tracking Controls on $\SO$}

\author{Taeyoung Lee
\thanks{Taeyoung Lee, Mechanical and Aerospace Engineering, George Washington University, Washington DC 20052 {\tt tylee@gwu.edu}}
\thanks{This research has been supported in part by NSF under the grants CMMI-1243000 (transferred from 1029551), CMMI-1335008, and CNS-1337722.}
}

\theoremstyle{definition}
\newtheorem{definition}{Definition}

\newtheorem{prop}{Proposition}

\graphicspath{{./Figs/}}

\begin{document}
\allowdisplaybreaks
\maketitle 

\begin{abstract}
This paper presents four types of tracking control systems for the  attitude dynamics of a rigid body. First, a smooth control system is constructed to track a given desired attitude trajectory, while guaranteeing almost semi-global exponential stability. It is extended to achieve global exponential stability by using a hybrid control scheme based on multiple configuration error functions. They are further extended to obtain robustness with respect to a fixed disturbance using an integral term.  The resulting robust, global exponential stability for attitude tracking is the unique contribution of this paper, and these are developed directly on the special orthogonal group to avoid singularities of local coordinates, or ambiguities associated with quaternions. The desirable features are illustrated by numerical examples. 
\end{abstract}

\section{Introduction}


The attitude dynamics of a rigid body have been extensively studied under various assumptions~\cite{Hug86,WenKreITAC91}. One of the distinct features of the attitude dynamics is that it evolves on a nonlinear manifold, namely the three-dimensional special orthogonal group. This yields unique stability properties that cannot be observed from dynamic systems on a linear space. For example, it has been shown that there exists no continuous control system that asymptotically stabilizes an attitude globally~\cite{BhaBerSCL00}.


Such topological obstruction in attitude stabilization has been dealt with two distinct approaches. In~\cite{MaiBerITAC06,ChaSanICSM11}, smooth attitude control systems are designed, guaranteeing \textit{almost} global asymptotic stability, where the region of attraction excludes  only a set of zero measure. This can be considered as the strongest stability property for smooth attitude control systems. 
On the other hand, a hysteresis-based switching algorithm is introduced to achieve global asymptotic stability~\cite{MaySanITAC11,MayTeePACC11,SchLorA12}, and a similar approach has been developed for the spherical orientation of reduced attitude tracking in~\cite{MayTeeA13}. A switching algorithm with an almost non-increasing Lyapunov function is constructed for global asymptotic stability with underactuated control inputs~\cite{CasAstA08}. 
But these results are based on either LaSalle's principle or hybrid invariance principles, and therefore, they only guarantee asymptotic stability, and robustness with respect to uncertainties has not been addressed in achieving global attractiveness  in attitude controls.

Attitude control systems can also be categorized with the choice of attitude representation. It is well known that minimal attitude representations, such as Euler angles or modified Rodriguez parameters, suffer from singularities~\cite{StuSR64}. They are not suitable for large angle rotational maneuvers, as the type of representation should be switched frequently to avoid the region of singularities. Quaternions do not have singularities but, as the three-sphere double-covers the special orthogonal group, a single attitude may be represented by two antipodal points on the three-sphere. This ambiguity should be carefully resolved in quaternion-based attitude control systems~\cite{MaySanITAC11}, otherwise they may exhibit unwinding, where a rigid body unnecessarily rotates through a large angle even if the initial attitude error is small~\cite{BhaBerSCL00}. To avoid these, an additional mechanism to lift measurements of attitude onto the three-sphere is introduced~\cite{MaySanITAC11}.

In this paper, four types of attitude control systems are presented to follow a given desired attitude trajectory. A smooth attitude control system is developed for almost semi-global exponential stability, and a hybrid control system with a new form of direction-based configuration error functions is introduced for global exponential stability with simpler controller structures. Each of them is extended with a unique integral control term to achieve robust global exponential stability in the presence of disturbance.

The proposed attitude control systems have the following distinct features. First, they provide stronger exponential stability. The attitude control systems in the aforementioned papers rely on the invariance principle, or an exogenous system is introduced to reformulate a tracking problem into stabilization of an autonomous system~\cite{MaySanITAC11,MayTeePACC11}, thereby yielding asymptotic stability. In this paper, rigorous Lyapunov stability analysis is presented to guarantee stronger, uniform exponential stability for each of four attitude control systems.


Second, a new intuitive form of attitude configuration error functions is introduced to simplify the design of hybrid attitude control systems. Configuration error functions in the prior literature, such as~\cite{MayTeePACC11} are based on compositions with smooth operations representing stretched rotations, and it is not straightforward to obtain proper controller parameters such as a hysteresis gap for stability. In this paper, a family of configuration error functions is constructed by comparing the desired directions with the current directions, and they yield an explicit and compact form of stability criteria. This simplifies the procedure to design  hybrid control systems for global attitude tracking.

Third, a special form of integral term is proposed to achieve robustness with respect to disturbances. Nonlinear PID-like attitude control systems have been studied in~
\cite{SubJAS04,ShoJuaPACC02,SuCaiJGCD11}. %
But, either they have singularities~\cite{SubJAS04,ShoJuaPACC02}, or they are based on the invariance principle that is valid only for attitude stabilization
~\cite{SuCaiJGCD11}. 
The robust attitude controls presented in this paper yield global exponential stability for attitude tracking problems considered as time-varying systems, and they guarantee an exponential convergence of the error in estimating the disturbance, as well as the attitude tracking error variables.

Another distinct feature is that attitude control systems are developed directly on the special orthogonal group. Therefore, singularities or complexities associated with minimal representations are avoided. Also, the ambiguity of quaternions does not have to be addressed by an additional mechanism to avoid the unwinding. In short, the proposed attitude control systems have simpler controller structures, and they provide stronger exponential stability properties as well as robustness.

This paper is organized as follows. An attitude tracking problem is formulated at Section \ref{sec:SO}. In the absence of disturbances, a smooth attitude control system to achieve almost semi-global asymptotic stability and a hybrid attitude control to guarantee global exponential stability are presented at Section \ref{sec:3}, respectively. They are extended with consideration of disturbance at Section \ref{sec:RAT}, which is followed by numerical examples and conclusions. Four types of the attitude control systems presented in this paper are summarized at Figure \ref{fig:FT}.

\begin{figure}
\setlength{\unitlength}{0.045\textwidth}\footnotesize
\begin{picture}(10,4.2)(-6,-1.4)
\put(-6,0.9){\shortstack[c]{Without\\ disturbance\\ $\Delta=0$}}
\put(-3.3,2.3){\shortstack[c]{Smooth control}}
\put(-4.3,0.6){\framebox(4,1.4)[c]{\shortstack[c]{Smooth Attitude Tracking\\(Section \ref{sec:31})}}}
\put(-0.3,1.3){\vector(1,0){0.8}}
\put(1.6,2.3){\shortstack[c]{Hybrid control}}
\put(0.5,0.6){\dashbox{0.08}(4,1.4)[c]{\shortstack[c]{Hybrid Attitude Tracking\\
(Section \ref{sec:GES})}}}
\put(-2.3,0.6){\vector(0,-1){0.6}}
\put(2.5,0.6){\vector(0,-1){0.6}}
\put(-6,-1.1){\shortstack[c]{With\\ disturbance\\ $\Delta\neq0$}}
{\linethickness{1.0pt}
\put(-4.3,-1.4){\framebox(4,1.4)[c]{\shortstack[c]{Smooth Attitude Tracking\\with Disturbance\\ (Section \ref{sec:RAG})}}}}
{\linethickness{1.0pt}
\put(0.5,-1.4){\dashbox{0.08}(4,1.4)[c]{\shortstack[c]{Hybrid Attitude Tracking\\with Disturbance\\ (Section \ref{sec:RGES})}}}}
\end{picture}
\caption{Four types of attitude tracking controls studied in this paper}\label{fig:FT}
\end{figure}
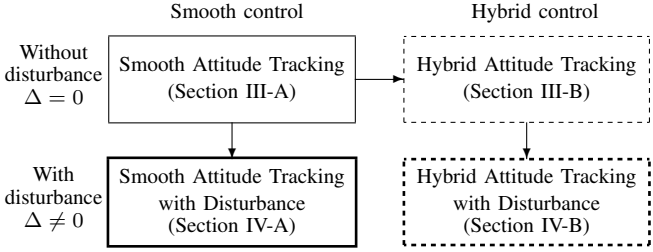

\newcommand{\mb}{\mathbf{m}}

\section{Problem Formulation}\label{sec:SO}

\subsection{Attitude Dynamics}

Consider the attitude dynamics of a rigid body. Define an inertial reference frame and a body-fixed frame. Its configuration manifold is the three-dimensional special orthogonal group:
$\SO = \{R\in\Re^{3\times 3}\,|\, R^TR=I,\;\mathrm{det}[R]=1\}$,
where a rotation matrix $R\in\SO$ represents the transformation of a vector from the body-fixed frame to the inertial reference frame. The equations of motion are given by
\begin{gather}
J\dot\Omega + \Omega\times J\Omega = u+\Delta,\label{eqn:Wdot}\\
\dot R = R\hat\Omega=\hat\omega R,\label{eqn:Rdot}
\end{gather}
where $J\in\Re^{3\times 3}$ is the inertia matrix, and $\Omega\in\Re^3$ is the angular velocity represented with respect to the body-fixed frame. We have $\omega=R\Omega$ that is the angular velocity represented with respect to the inertial frame. The control moment and the unknown, but fixed uncertainty are denoted by $u\in\Re^3$ and $\Delta\in\Re^3$, respectively. It is assumed that the fixed uncertainty is bounded by a known constant $B_\Delta\in\Re$ as
\begin{align}
\|\Delta\| \leq B_\Delta.\label{eqn:Bd}
\end{align}

At \refeqn{Rdot}, the \textit{hat} map $\wedge :\Re^{3}\rightarrow\so$ represents the transformation of a vector in $\Re^3$ to a $3\times 3$ skew-symmetric matrix such that $\hat x y = x\times y$ for any $x,y\in\Re^3$~\cite{BulLew05}. 
More explicitly, 
\begin{align*}
\hat x = \begin{bmatrix} 0 & -x_3 & x_2 \\ x_3 & 0 & -x_1 \\ -x_2 & x_1 & 0\end{bmatrix},
\end{align*}
for $x=[x_1,x_2,x_3]^T\in\Re^3$. 
In some cases, $\hat x y$ is written as $(x)^\wedge y$ for conciseness. The inverse of the hat map is denoted by the \textit{vee} map $\vee:\so\rightarrow\Re^3$. Several properties of the hat map used in this paper are summarized as
\begin{gather}
    x\cdot \hat y z = y\cdot \hat z x,\quad \hat x\hat y z = (x\cdot z) y - (x\cdot y ) z\label{eqn:STP},\\
    \widehat{x\times y} = \hat x \hat y -\hat y \hat x = yx^T-xy^T,\label{eqn:hatxy}\\
R\hat x R^T = (Rx)^\wedge,\quad 
R(x\times y) = Rx\times Ry\label{eqn:RxR}
\end{gather}
for any $x,y,z\in\Re^3$ and $R\in\SO$. Throughout this paper, the standard dot product in $\Re^3$ is denoted as $x\cdot y = x^T y$ for any $x,y\in\Re^n$, and the maximum eigenvalue and the minimum eigenvalue of $J$ are denoted by $\lambda_M$ and $\lambda_m\in\Re$, respectively.

\subsection{Attitude Tracking Problem}

The two-sphere is the manifold of unit-vectors in $\Re^3$, i.e., $\Sph^2=\{q\in\Re^3\,|\, \|q\|=1\}$. Let $b_1,b_2\in\Sph^2=\{q\in\Re^3\,|\, \|q\|=1\}$ be the unit-vectors from the mass center of the rigid body toward two distinct, characteristic points on the rigid body, represented with respect to the body-fixed frame. For example, they may represent the direction of the optical axis for an onboard vision-based sensor. Due to the rigid body assumption, we have $\dot b_1=\dot b_2=0$. Without loss of generality, we assume that $b_1$ is normal to $b_2$, i.e., $b_1\cdot b_2=0$.  If $b_1\cdot b_2\neq 0$, we choose a fictitious $b_3$ as $b_3=\frac{b_1\times b_2}{\|b_1\times b_2\|}$, and rename it as $b_2$. 

From now on, the subscript $i$ is assumed to be $i\in\{1,2\}$.  Let $r_i\in\Sph^2$ be the representation of $b_i$ with respect to the inertial frame, i.e., $r_i = R b_i$. 
Note that $r_i$ may change over time as the rigid body rotates even though $b_i$ is fixed.  Using \refeqn{Rdot}, the kinematics equation for $r_i$ is given by
\begin{align}
\dot r_i = \hat\omega R b_i = \omega\times r_i.
\end{align}



Suppose that a smooth desired attitude trajectory is given by $R_d(t)$, and it  satisfies the following kinematics equation:
\begin{align}
\dot R_d = \hat\omega_d R_d,\label{eqn:Rd}
\end{align}
where $\omega_d\in\Re^3$ is the desired angular velocity expressed in the inertial frame. It is assumed that the desired angular velocity and its derivatives are uniformly bounded. Next, we transform the desired attitude into the desired directions for $r_i$ as
\begin{align}
r_{i_d} = R_d b_i.\label{eqn:rdi}
\end{align}
The desired directions $r_{i_d}\in\Sph^2$ are used to construct a new form of configuration error functions for  hybrid control systems developed later. 
From \refeqn{Rd} and \refeqn{rdi}, we have the following kinematics equation,
\begin{align}
\dot r_{i_d} = \omega_{d}\times r_{i_d}.\label{eqn:rddot}
\end{align}
and they are consistent with the rigid body assumption, i.e., $b_1\cdot b_2 = r_1\cdot r_2 = r_{1_d}\cdot r_{2_d} =0$.
%
%
The goal is to design a control input $u$ such that the attitude $R=R_d$ becomes an exponentially stable equilibrium of the controlled system. 


\section{Attitude Tracking with No Disturbance}\label{sec:3}


In this section, we assume that there is no disturbance, i.e., $\Delta=0$. A smooth control system is first developed for almost semi-global exponentially stability, and a hybrid control system with new set of configuration error functions is proposed for global attitude tracking.

\subsection{Almost Global Attitude  Tracking}\label{sec:31}

Error variables are defined to represent the difference between the desired directions $r_{i_d}$ and the current directions $r_i=Rb_i$. Define the $i$-th configuration error function as
\begin{align}
\Psi_i(R) =\frac{1}{2}\|Rb_i-r_{i_d}\|^2 = 1 - Rb_i\cdot r_{i_d},\label{eqn:Psii}
\end{align}
which represents $1-\cos\theta_i$, where $\theta_i$ is the angle between $Rb_i$ and $r_{i_d}$. Therefore, it is positive definite about $Rb_i=r_{i_d}$ where $\theta_i=0$, and the critical points are given by $Rb_i=\pm r_{i_d}$. 
The $i$-th configuration error vector is defined as
\begin{align}
e_{r_i} & = R^T r_{i_d}\times b_i.\label{eqn:eri}
\end{align}
For positive constants $k_1\neq k_2$, we also define the complete configuration error function and error vector as
\begin{align}
\Psi(R) &= k_1 \Psi_1(R) + k_2 \Psi_2(R),\label{eqn:PsiSO}\\
e_r  &= k_1 e_{r_1} + k_2 e_{r_2}.\label{eqn:er}
\end{align}
The angular velocity error vector is defined as
\begin{align}
e_\Omega = \Omega - R^T\omega_d. \label{eqn:eW}
\end{align}

\begin{prop}\label{prop:errSO}
The error variables \refeqn{Psii}-\refeqn{eW}, representing the difference between the solution of the equations of motion \refeqn{Wdot} and \refeqn{Rdot}, and the given desired trajectory \refeqn{rdi} with \refeqn{rddot}, satisfy the following properties. For $i\in\{1,2\}$,
\renewcommand{\labelenumi}{(\roman{enumi})}
\begin{enumerate}\renewcommand{\itemsep}{3pt}
\item $\dot \Psi_i(R)=e_{r_i}\cdot e_\Omega$, and $\dot\Psi(R)= e_r\cdot e_\Omega$.
\item $\|\dot e_{r_i}\| \leq \|e_\Omega\|$, and $\|\dot e_r\|\leq (k_1+k_2)\|e_\Omega\|$.
\item Let $h_1 = 2\min\{k_2,k_1 \}$, $h_2 = 4\max\{(k_1-k_2)^2,k_2^2,k_1^2\}$,
$h_3 = 4\max\{(k_1+k_2)^2,k_2^2,k_1^2\}$, 
$h_4 = 2(k_1+k_2)$,
$h_5 = 4\min\{(k_1+k_2)^2,k_2^2,k_1^2\}$, and let $\psi$ be a constant satisfying $0<\psi<h_1$. Then, we have
%
%
%
\begin{align}
\frac{h_1}{h_2+h_3}\|e_r\|^2  \leq \Psi(R) \leq \frac{h_1h_4}{h_5(h_1-\psi)}\|e_r\|^2,\label{eqn:PsibSO}
\end{align}
where the upper bound is satisfied when $\Psi(R)\leq \psi$.


\end{enumerate}
\end{prop}
\begin{proof}
See Appendix \ref{sec:errSO}.
\end{proof}

Two stability concepts are introduced as follows. 
\begin{definition}
Consider an equilibrium of a dynamic system located at the origin. The equilibrium is 
\begin{itemize}
\item[(i)] \textit{almost globally asymptotically stable}, if it is asymptotically stable and almost all trajectories converge to it, i.e., the set of the initial states that do not asymptotically converge to the origin has zero Lebesgue measure. 
\item[(ii)] \textit{almost semi-globally exponentially stable}, if it is asymptotically stable, and for almost all initial states, there exist finite controller gains or parameters such that the corresponding trajectory exponentially converges to the origin, i.e., the set of the initial states that cannot not exponentially converge to the origin has zero Lebesgue measure.
\end{itemize}
\end{definition}
\noindent The concept of almost global stability appears in~\cite{MonITAC03,RanSCL01}, and it has been applied to smooth attitude control systems, such as \cite{MaiBerITAC06,ChaSanICSM11}, since it is impossible to achieve global attractivity on $\SO$ due to the topological restriction~\cite{BhaBerSCL00}. Almost semi-globally exponential stability implies that controller parameters can be chosen such that exponential stability is guaranteed for almost all trajectories. The control system presented in this paper guarantees the above two stability properties. 

\begin{prop}\label{prop:AGASSO}
Consider the dynamic system \refeqn{Wdot}, \refeqn{Rdot} with $\Delta=0$. A desired trajectory is given by \refeqn{rddot}. For $k_1,k_2,k_\Omega >0$ with $k_1\neq k_2$, a control input is chosen as
\begin{align}
u & = -e_r -k_\Omega e_\Omega  +(R^T\omega_d)^\wedge JR^T\omega_d
+ JR^T\dot\omega_d.\label{eqn:uSO}
\end{align}
Then, the following properties hold:
\renewcommand{\labelenumi}{(\roman{enumi})}
\begin{enumerate}
\item The set of equilibrium points is given by $\{(R,\omega)\in \SO\times\Re^3\,|\,(R_d,\omega_d), (\exp(\pi \hat r_{1_d})R_d,\omega_d)$, $(\exp(\pi \hat r_{2_d})R_d,\omega_d), (\exp(\pi (r_{1_d}\times r_{2_d})^\wedge)R_d,\omega_d)\}$.
\item The desired equilibrium $(R_d,\omega_d)$ is almost globally asymptotically stable and almost semi-globally exponentially stable, i.e., the set of the following initial conditions that guarantee exponential stability almost cover $\SO\times\Re^3$ when $k_1,k_2$ are sufficiently large:
\begin{gather}
\Psi(R(0)) \leq \psi < 2\min\{k_2,k_1 \},\label{eqn:Psi0SO}\\
\begin{aligned}
& \|e_\Omega(0)\|^2  \leq \frac{2}{\lambda_M}(\psi-\Psi(R(0))).
\end{aligned}\label{eqn:eW0}
\end{gather}
\item The three undesired equilibria are unstable.
\end{enumerate}
\end{prop}

\begin{proof}
Using \refeqn{eW}, \refeqn{uSO} and rearranging, the time-derivative of $Je_\Omega$ can be written as
\begin{align}
J\dot e_\Omega & = d\times e_\Omega - e_r-k_\Omega e_\Omega,\label{eqn:JeWdot}
\end{align}
where $d=Je_\Omega + (2J-\trs{J}I)R^T\omega_d\in\Re^3$.


The equilibrium corresponds to the critical points of $\Psi(R)$ where its derivatives become zero, i.e., $r_1=\pm r_{1_d}$ or $r_2=\pm r_{2_d}$, and $\omega=\omega_d$. For example, when $r_1=r_{1_d}$ and $r_2=r_{2_d}$, we have $R=R_d$. When $r_1=-r_{1_d}$ and $r_2=r_{2_d}$, the attitude is the $180^\circ$ rotation of $R_d$ about $r_{2_d}$, yielding $R=\exp(\pi r_{2_d}) R_d$. Other equilibria are obtained similarly, and these show (i).

Let a Lyapunov candidate function be
\begin{align*}
\mathcal{V} = \frac{1}{2} e_\Omega \cdot J e_\Omega + \Psi
 + c Je_\Omega\cdot e_r
\end{align*}
for a positive constant $c$. Using \refeqn{PsibSO}, we can show that
\begin{align}
z^T M_1 z \leq \mathcal{V},\label{eqn:Vlb}
\end{align}
where $z=[\|e_r\|,\|e_\Omega\|]^T\in\Re^2$ and $M_1\in\Re^{2\times 2}$ is given by
\begin{align}
M_1 = \frac{1}{2}\begin{bmatrix} \frac{2h_1}{h_2+h_3} & -c\lambda_M \\-c\lambda_M & \lambda_m\end{bmatrix}.\label{eqn:M1}
\end{align}

From \refeqn{JeWdot} and the property (i) of Proposition \ref{prop:errSO}, 
\begin{align*}
\dot{\mathcal{V}} 
& = - k_\Omega \|e_\Omega\|^2
+ cJe_\Omega\cdot \dot e_r
+cJ\dot e_\Omega \cdot e_r .
\end{align*}
We find the bound of the last two terms of the above equation. Using the property (ii) of Proposition \ref{prop:errSO}, we have 
\begin{align*}
 Je_\Omega \cdot \dot e_r \leq \lambda_M(k_1+k_2) \|e_\Omega\|^2.
\end{align*}
As the desired angular velocity is bounded by the assumption, there exists a constant $B>0$ satisfying
\begin{align*}
\|(2J-\tr{J}I)R^T\omega_d\| \leq 
\|(2J-\tr{J}I)\| \|\omega_d\| \leq 
B.
\end{align*}
From \refeqn{JeWdot} and using the fact that $\|e_{r}\|\leq k_1+k_2$, we have
\begin{align*}
&J\dot e_\Omega\cdot e_r \leq \lambda_M (k_1+k_2)\|e_\Omega\|^2 + (B+k_\Omega) \|e_\Omega\|\|e_r\| -\|e_r\|^2.
\end{align*}
From these, an upper bound of $\dot{\mathcal{V}}$ can be written as 
\begin{align}
\dot{\mathcal{V}} \leq -z^T M_3 z,\label{eqn:Vdot}
\end{align}
where the matrix $M_3\in\Re^{2\times 2}$ is given by
\begin{align}
M_3 = 
\begin{bmatrix}
c
& -\frac{c(B+k_\Omega)}{2}\\
-\frac{c(B+k_\Omega)}{2}
& k_\Omega-2c(k_1+k_2)\lambda_M
\end{bmatrix}.\label{eqn:M3}
\end{align}

If the constant $c$ is chosen sufficiently small such that
\begin{align}
c&<\min\Big\{\sqrt{\frac{2\lambda_m h_1}{\lambda_M^2h_{23}}},
 \frac{4 k_\Omega}{8k\lambda_M+(B+k_\Omega)^2}\Big\},\label{eqn:c}
\end{align}
where $h_{23}=h_2+h_3$, $k=k_1+k_2$, then the matrices $M_1,M_3$ are positive definite, which shows that the desired equilibrium is asymptotically stable, and $e_r,e_\Omega\rightarrow 0$ as $t\rightarrow\infty$.

However, the fact that $e_r\rightarrow 0$ does not necessarily imply that $R\rightarrow R_d$ as $t\rightarrow \infty$, since $e_r=0$ also at three undesired equilibria. Therefore, we cannot achieve global asymptotic stability for the given control system. Instead, we show almost global asymptotic stability as follow. At the first undesired equilibrium given by $R=\exp(\pi\hat r_{1_d})R_d$ and $e_\Omega=0$, we have $\mathcal{V}= 2 k_2$. Define
%
\begin{align*}
\mathcal{W} = 2k_2 - \mathcal{V}= -\frac{1}{2}e_\Omega\cdot J e_\Omega +(2k_2-\Psi) - c e_r\cdot e_\Omega.
\end{align*}
Then, $\mathcal{W}=0$ at the undesired equilibrium. We have
\begin{align*}
\mathcal{W} \geq -\frac{\lambda_M}{2}\|e_\Omega\|^2 +(2k_2-\Psi) - c \|e_r\|\| e_\Omega\|.
\end{align*}
Due to the continuity of $\Psi$, we can choose $R$ that is arbitrary close to $\exp(\pi\hat r_{1_d})R_d$ such that $(2k_2-\Psi)>0$. Therefore, if $\|e_\Omega\|$ is sufficiently small, we obtain $\mathcal{W}>0$ at those points. In other words, at any arbitrarily small neighborhood of the undesired equilibrium, there exists a domain in which $\mathcal{W} >0$, and we have $\dot{\mathcal{W}}= -\dot{\mathcal{V}} > 0$ from \refeqn{Vdot}. According to Theorem 4.3 in~\cite{Kha02}, the undesired equilibrium is unstable. The instability of the other two equilibrium configurations can be shown by the similar way. This shows (iii).

The region of attraction to the desired equilibrium excludes the stable manifolds to the undesired equilibria. But the dimension of the union of the stable manifolds to the unstable equilibria is less than the tangent bundle of $\SO$. Therefore, the measure of the stable manifolds to the unstable equilibrium is zero. Then, the desired equilibrium is \textit{almost} globally asymptotically stable~\cite{ChaSanICSM11}.

Next, we show exponential stability. Define $\mathcal{U} = \frac{1}{2} e_\Omega \cdot J e_\Omega + \Psi$. From \refeqn{JeWdot} and the property (i) of Proposition \ref{prop:errSO}, we have $\dot{\mathcal{U}} = - k_\Omega \|e_\Omega\|^2$, 
which implies that $\mathcal{U}(t)$ is non-increasing. For the initial conditions satisfying \refeqn{Psi0SO} and \refeqn{eW0}, we have $\mathcal{U}(0)\leq \psi$. Therefore, we obtain
\begin{align}
\Psi(R(t)) & \leq \mathcal{U}(t) \leq \mathcal{U}(0) \leq \psi < 2\min\{k_2,k_1 \}\label{eqn:Psitb}.
\end{align}
Thus, the upper bound of \refeqn{PsibSO} is satisfied. This yields
\begin{align}
\mathcal{V} \leq z^T M_2 z,\label{eqn:Vub}
\end{align}
where the matrix $M_2$ is given by
\begin{align}
M_2 = \frac{1}{2}\begin{bmatrix} \frac{2h_1h_4}{h_5(h_1-\psi)} & c\lambda_M \\c\lambda_M & \lambda_M\end{bmatrix}.\label{eqn:M2}
\end{align}
The condition on $c$ given by \refeqn{c} also guarantees that $M_2$ is positive definite. Therefore, from \refeqn{Vlb}, \refeqn{Vdot}, and \refeqn{Vub}, the desired equilibrium is exponentially stable~\cite{Kha02}. The initial attitudes $R(0)$ satisfying \refeqn{Psi0SO} almost cover $\SO$ as $k_1\rightarrow k_2$, excluding only three attitudes of undesired equilibria, and the initial angular velocities $\Omega(0)$ satisfying \refeqn{eW0}  cover $\Re^3$ as $k_1,k_2\rightarrow \infty$. In short, the set of initial conditions $(R(0),\Omega(0))$ that guarantee exponential stability almost cover $\SO\times\Re^3$ as $k_1\rightarrow\infty$ and $k_2\rightarrow k_1$. Therefore, the desired equilibrium is almost semi-globally exponentially stable. 
\end{proof}

Compared with other attitude control systems achieving almost global asymptotic stability for attitude stabilization of time-invariant systems on \SO, such as~\cite{ChaSanICSM11}, this proposition guarantees stronger almost semi-global \textit{exponential} stability for attitude \textit{tracking} of time-varying systems. 


The fact that the region of attraction does not cover the entire configuration manifold  is not a major issue in practice, as the probability that a given initial condition exactly lies in the stable manifolds to the unstable equilibria is zero, provided that the initial condition is randomly chosen. But, the existence of such stable manifolds may have strong effects on the dynamics of the controlled system~\cite{LeeLeoPICDC11}. In particular, the proportional term of the control input, namely $e_{r_i}$ approaches zero as the attitude becomes closer to one of the three undesired equilibria, thereby causing a slow convergence rate especially for large attitude errors. In the following subsection, discontinuities are introduced in the control input to achieve global exponential stability with improved convergence rates.


\subsection{Hybrid Control for Global Attitude Tracking}\label{sec:GES}

Recently, hybrid control systems for global attitude stabilization are developed in terms of quaternions~\cite{MaySanITAC11}, and rotation matrices~\cite{MayTeePACC11}, respectively. The key idea is switching between different forms of configuration error functions, referred to as synergistic potential functions, such that the attitude is expelled from the vicinity of undesired equilibria. The switching logic is defined with a hysteresis model to improve robustness with respect to measurement noises. This paper follows the same framework, but a new form of synergistic configuration error functions is provided to simplify controller structures and controller design procedure. The given control system also provides stronger global exponential stability that is uniformly applied to time-varying systems for tracking problems.

We first introduce a mathematical formulation of hybrid systems~\cite{GoeSanICSM09}. Let $\mathcal{M}$ be the set of discrete modes, and let $\mathcal{Q}$ be the domain of continuous states. Given a state $(\mb,\xi)\in\mathcal{M}\times \mathcal{Q}$, a hybrid system is defined by
\begin{alignat}{2}
\dot \xi & = \mathcal{F}(\mb,\xi),&\quad (\mb,\xi)&\in\mathcal{C},\label{eqn:Hyb1}\\
\mb^+ & = \mathcal{G}(\mb,\xi),& (\mb,\xi)&\in\mathcal{D},\label{eqn:Hyb2}
\end{alignat}
where the flow map $\mathcal{F}:\mathcal{M}\times\mathcal{Q}\rightarrow\T\mathcal{Q}$ describes the evolution of the continuous state $\xi$; the flow set $\mathcal{C}\subset\mathcal{M}\times\Re^n$ defines where the continuous state evolves; the jump map $\mathcal{G}:\mathcal{M}\times\mathcal{Q}\rightarrow\mathcal{M}$ governs the discrete dynamics; the jump set $\mathcal{D}\subset\mathcal{M}\times\mathcal{Q}$ defines where discrete jumps are permitted. 

For the proposed hybrid attitude control system, there are a nominal mode and two expelling modes. The control input at the nominal mode is equal to \refeqn{uSO} which is constructed by the following configuration error functions given at \refeqn{Psii}:
\begin{align}
\Psi_{N_i}(R)& =1-Rb_i\cdot r_{i_d},
\end{align}
where the subscript $N$ is used to explicitly denote that it is for the nominal mode, i.e., $\Psi_{N_i}\triangleq\Psi_i$. When the attitude becomes closer to undesired equilibria, the error function is switched to one of the following expelling error functions:
\begin{align}
\Psi_{E_1} (R) &= \alpha + \beta Rb_1\cdot (r_{1_d}\times r_{2_d}),\label{eqn:PsiE1}\\
\Psi_{E_2} (R) &= \alpha + \beta Rb_2\cdot (r_{1_d}\times r_{2_d}),\label{eqn:PsiE2}
\end{align}
for constant $\alpha,\beta$ satisfying $1<\alpha<2$ and $|\beta|< \alpha-1$.

For example, if the attitude becomes close to the critical point of the first nominal error function $\Psi_{N_1}$ where $Rb_1=-r_{1_d}$, the expelling configuration error $\Psi_{E_1}$ is engaged such that $Rb_1$ is steered toward a direction normal to $-r_{1_d}$, namely $-\frac{\beta}{|\beta|}(r_{1_d}\times r_{2_d})$, to rotate the rigid body away from the undesired critical point. Similarly, the second expelling configuration error function $\Psi_{E_2}$ is engaged near the critical points of $\Psi_{N_2}$. As a result, there are three discrete modes, namely $\mathcal{M}=\{\RI,\RII,\RIII\}$, and the configuration error function for each mode is given by
\begin{align}
\Psi_\RI(R) = k_1 \Psi_{N_1}(Rb_1) + k_2 \Psi_{N_2}(Rb_2),\label{eqn:PsiI}\\
\Psi_\RII(R) = k_1 \Psi_{N_1}(Rb_1) + k_2 \Psi_{E_2}(Rb_2),\label{eqn:PsiII}\\
\Psi_\RIII(R) = k_1 \Psi_{E_1}(Rb_1) + k_2 \Psi_{N_2}(Rb_2).\label{eqn:PsiIII}
\end{align}
In short, the nominal control input is constructed from the nominal error function $\Psi_\RI$. If the attitude is in the vicinity of the undesired critical points of $\Psi_{N_1}$ or $\Psi_{N_2}$, the control input is switched into the mode $\RIII$ or $\RII$, respectively.

The switching logic is formally specified as follows. Define a variable $\rho$ representing the minimum configuration error:
\begin{align}
\rho(R) = \min_{\mb\in\mathcal{M}} \{\Psi_\mb (R)\}.
\end{align}
Observing that the values of $\Psi_{N_1},\Psi_{N_2}$ are maximized at their undesired critical points, the jump map is chosen such that the discrete mode is switched to the new mode where the configuration error is minimum:
\begin{align}
\mathcal{G}(R) 
& = \argmin_{\mb\in\mathcal{M}} \{\Psi_\mb (R)\}
 = \{\mb\in\mathcal{M}\,:\, \Psi_\mb = \rho\}.\label{eqn:GSO}
\end{align}

It is possible to switch whenever a new  mode with a smaller value of configuration error function is available, or equivalently, when $\mathcal{G}(R)\neq \mb$ or $\Psi_\mb - \rho > 0$. However, the resulting controlled system may yield chattering due to measurement noise. Instead, a hysteresis gap $\delta$ is introduced for robustness, and a switching occurs if the difference between the current configuration error and the minimum value is greater than a prescribed hysteresis gap. More explicitly, the jump set and the flow set are given by
\begin{align}
\mathcal{D} & = \{ (R,\Omega,\mb) : \Psi_\mb - \rho \geq \delta \text{ and } \|e_\Omega\|\leq B_{e_\Omega}\},\label{eqn:DSO}\\
\mathcal{C} & = \{ (R,\Omega,\mb) : \Psi_\mb - \rho \leq \delta \text{ or } \|e_\Omega\|\geq B_{e_\Omega}\},\label{eqn:CSO}
\end{align}
for a positive constant $\delta$ that is specified later at \refeqn{deltaSO}, and an arbitrary positive constant $B_{e_\Omega}$. The condition on $e_\Omega$ is imposed to explicitly guarantee that the Lyapunov function used in the stability analysis strictly decreases over any jump. 

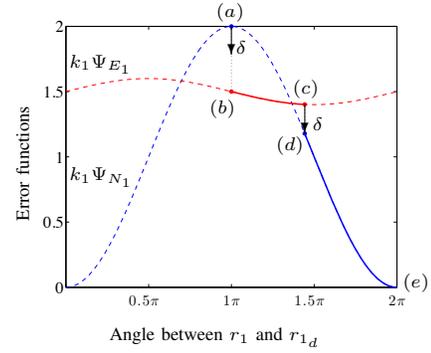
\begin{figure}
\setlength{\unitlength}{0.1\columnwidth}\scriptsize
\centerline{
\begin{picture}(5.5,5.3)(0,-0.5)
\put(0,0){\includegraphics[width=0.55\columnwidth]{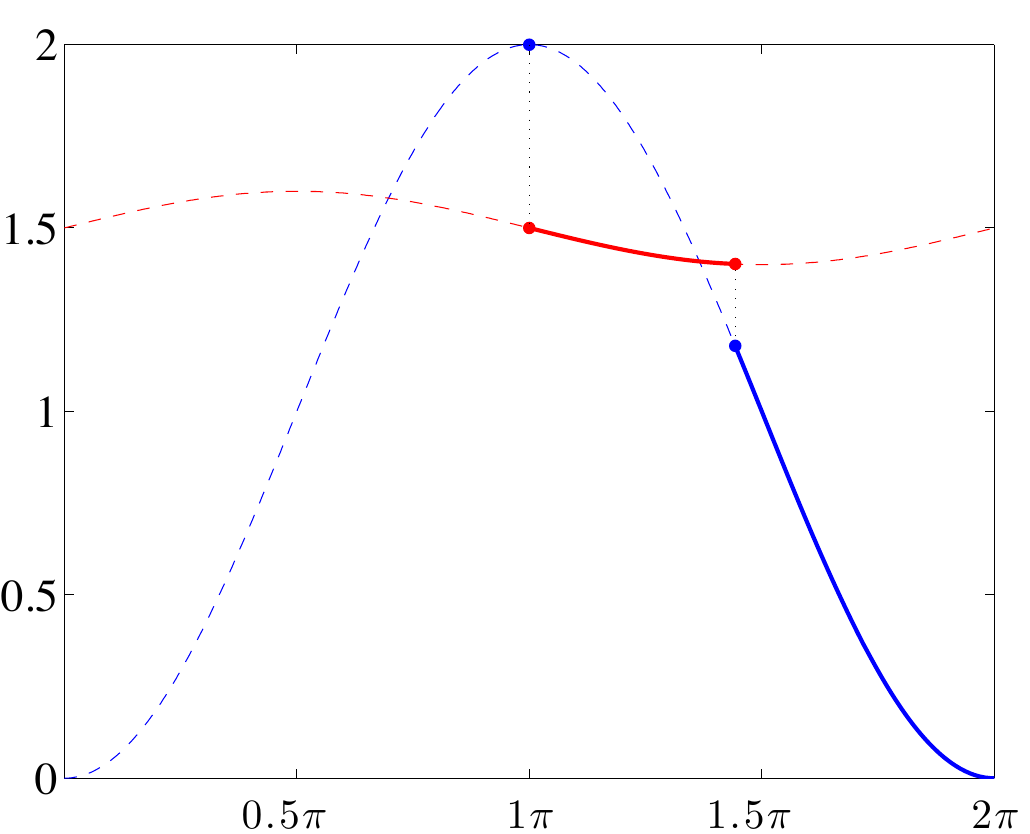}}
\put(1.0,-0.5){Angle between $r_1$ and $r_{1_d}$}
\put(-0.4,1.3){\rotatebox{90}{Error functions}}
\put(0.4,1.9){$k_1\Psi_{N_1}$}
\put(0.4,3.6){$k_1\Psi_{E_1}$}
\put(3.93,3.04){\vector(0,-1){0.43}}
\put(2.828,4.22){\vector(0,-1){0.43}}
\put(2.6,4.35){$(a)$}
\put(2.5,2.9){$(b)$}
\put(3.75,3.2){$(c)$}
\put(3.5,2.35){$(d)$}
\put(5.4,0.3){$(e)$}
\put(2.9,3.8){$\delta$}
\put(4.05,2.65){$\delta$}
\end{picture}}
\caption{Illustration of switching algorithm: Consider a trajectory starting from $(a)$ that is an undesired equilibrium of $k_1\Psi_{N_1}$. Since $k_1\Psi_{N_1}-k_1\Psi_{E_1} >\delta$, it is switched to $(b)$ and it moves right to reduce $k_1\Psi_{E_1}$. Once $k_1\Psi_{E_1}-k_1\Psi_{N_1}\geq\delta$ at $(c)$, it is switched to $(d)$ on $k_1\Psi_{N_1}$ and it moves right until the error becomes zero at $(e)$. In short, by switching to the expelling error function $k_1\Psi_{E_1}$ temporarily between (b) and (c), it avoids the undesired equilibrium $(a)$ of $k_1\Psi_{N_1}$.}
\end{figure}

The control input at each mode is constructed from the corresponding configuration error function by following the same procedure described in the previous section:
\begin{align}
u & = -e_H -k_\Omega e_\Omega  +(R^T\omega_d)^\wedge JR^T\omega_d
+ JR^T\dot\omega_d.\label{eqn:uHSO}
\end{align}
where the hybrid configuration error vectors are defined as
\begin{align}
e_H & = k_1 e_{H_1} + k_2 e_{H_2},\label{eqn:eH}\\
e_{H_1} & =
\begin{cases}
e_{r_1}  &\text{ if $\mb = \RI,\RII$},\\
-\beta R^T (r_{1_d}\times r_{2_d})\times b_1 &\text{ if $\mb = \RIII$},
\end{cases}\label{eqn:eH1}\\
e_{H_2} & =
\begin{cases}
e_{r_2} &\text{ if $\mb = \RI,\RIII$},\\
-\beta R^T (r_{1_d}\times r_{2_d})\times b_2 &\text{ if $\mb = \RII$}.
\end{cases}\label{eqn:eH2}
\end{align}

Exponential stability of hybrid systems evolving on $\Re^n$ has been introduced in~\cite{TeeForITAC13} by defining a distance between a set and a state in terms of the Euclidean norm. Generalizing the concept of exponential stability formally to arbitrary hybrid systems evolving on a nonlinear manifold is out of scope of this paper. Instead, we use the property of the proposed hybrid control system that the error variable $e_H$ may become zero only at the desired, nominal mode, i.e., only when $\mb=\RI$, and exponential stability is considered as imposing an exponential bound on the selected error variables of the continuous states as follows.



\begin{definition}\label{def:ES}
Let $a\in\mathcal{M}\times\mathcal{Q}$ be an equilibrium of \refeqn{Hyb1} and \refeqn{Hyb2}. Suppose $e:\mathcal{M}\times\mathcal{Q}\rightarrow \Re^q$ be an error variable satisfying $\|e\|=0$ at $a$, and $\|e\|\neq 0$ otherwise, where $q$ is the dimension of $\mathcal{Q}$. The equilibrium $a$ is \textit{globally exponentially stable with respect to $e$}, if it is globally asymptotically stable, and there exist $\lambda_0,\lambda_1>0$ such that $\|e(t)\|\leq \lambda_0 \|e(0)\| \exp(-\lambda_1 t)$ for any $e(0)$ and all $t\geq 0$.
\end{definition}

\begin{prop}\label{prop:GES}
Consider a hybrid control system defined by \refeqn{PsiI}-\refeqn{eH2}. For given constants $k_1,k_2,\alpha,\beta$ satisfying $k_1,k_2>0$, $k_1\neq k_2$, $1<\alpha<2$ and $|\beta|< \alpha-1$, choose the hysteresis gap $\delta$ such that
\begin{align}
0<\delta < \min\{k_1,k_2\}\min\{2-\alpha, \alpha-|\beta|-1\}.\label{eqn:deltaSO}
\end{align}
Then, the desired equilibrium $(R_d,\omega_d)$ is globally exponentially stable with respect to $z=[\|e_H\|,\, \|e_\Omega\|]\in\Re^2$.
\end{prop}

\begin{proof}
The set of values for $(r_1,r_2)=(Rb_1,Rb_2)$ at the critical points of each configuration error function is given by $\mathcal{R}_I = \{(r_1,r_2)\,|\, (\pm r_{1_d},\pm r_{2_d})\}$, 
$\mathcal{R}_{II} = \{(r_1,r_2)\,|\, (\pm r_{1_d},\pm r_{1_d}\times r_{2_d})\}$, 
$\mathcal{R}_{III} = \{(r_1,r_2)\,|\, (\pm r_{1_d}\times r_{2_d},\pm r_{2_d})\}$.
Therefore, there are twelve critical points in total, including the desired equilibrium $(r_{1_d},r_{2_d})$, and eleven undesired critical points. 

We first show that the undesired critical points cannot become an equilibrium of the controlled system as they belong to the jump set $\mathcal{D}$. At the first undesired critical point of $\Psi_\RI$, namely $(r_1,r_2)=(r_{1_d},-r_{2_d})$, we have
\begin{align*}
\Psi_\RI	=	2k_2,\quad 
\Psi_{\RII} = \alpha k_2,\quad
\Psi_{\RIII} = \alpha k_1+2k_2.
\end{align*}
which gives $\rho=\min_{\mb}\Psi_{\mb}=\alpha k_2$ as $\alpha <2$. This yields $\Psi_\RI-\rho=(2-\alpha) k_2 \geq \delta$ from the definition of $\delta$ given at \refeqn{deltaSO}. Therefore, the first critical point corresponding to $(r_1,r_2)=(r_{1_d},-r_{2_d})$ with $e_\Omega=0$ lies in the jump set $\mathcal{D}$. 
This can be repeated to show that all of the undesired critical points of the configuration error functions belong to the jump set. Thus, the desired equilibrium is the only equilibrium of the controlled system.

The remaining part of the proof is similar to the proof of Proposition \ref{prop:AGASSO}. For the nominal mode $\mb=\RI$, all of properties at Proposition \ref{prop:errSO} are automatically satisfied as the definitions of the configuration error function and error vectors are identical. Furthermore, since the flow set $\mathcal{C}$ excludes undesired critical points where $e_r=0$, there exists a constant $\gamma>0$ such that
\begin{align}
\Psi \leq \gamma \|e_r\|^2,\label{eqn:PsiBgam}
\end{align}
for any $R,R_d$ in the flow set $\mathcal{C}$. We can show the same properties for $\Psi_{\RII}$ and $\Psi_{\RIII}$. 


Define a Lyapunov function on $\mathcal{M}\times (\SO\times\Re^3)$ as
\begin{align*}
\mathcal{V}_\mb = \frac{1}{2} e_\Omega \cdot J e_\Omega + \Psi_\mb
 + c e_\Omega\cdot e_H.
\end{align*}
This is positive definite about the desired equilibrium at $\mb=\RI$. According to the above properties and the proof of Proposition \ref{prop:AGASSO}, $\mathcal{V}_\mb$ is positive definite and decrescent with respect to quadratic functions of $e_H$ and $e_\Omega$, and $\dot{\mathcal{V}}_\mb$ is less than a negative quadratic function of of $e_H$ and $e_\Omega$ in the flow set $\mathcal{C}$. Therefore, the error variables $e_H$ and $e_\Omega$ exponentially decrease in the flow set $\mathcal{C}$.

Note that the desired angular velocity $\omega_d$, therefore $e_\Omega$ does not change over any jump, since $\frac{d}{dt} (r_{1_d}\times r_{2_d}) = \omega_d\times (r_{1_d}\times r_{2_d})$. 
Therefore, the change of the Lyapunov function over the jump from a mode $\mb\in\mathcal{M}$ is
\begin{align*}
\mathcal{V}_\mathcal{G}-\mathcal{V}_\mb=\rho-\Psi_{\mb} + ce_\Omega\cdot e_H\big|^{\mathcal{G}}_\mb\leq -\delta + 2ck B_{e_\Omega},
\end{align*}
where we use the fact that $\|e_H\|\leq k_1+k_2 \triangleq k$ and \refeqn{DSO}. If the constant $c$ that is independent of the controller is chosen sufficiently small such that $c< \frac{\delta}{4ck B_{e_\Omega}}$, we have $\mathcal{V}_\mathcal{G}-\mathcal{V}_\mb< -\frac{\delta}{2}$, i.e., the Lyapunov function strictly decreases over any jump. It follows that the desired equilibrium is globally exponentially stable with respect to $z=[\|e_H\|,\, \|e_\Omega\|]$.
\end{proof}

The unique feature of the proposition is that it provides a stronger global exponential stability for a tracking problem on $\SO$, compared with the existing results in~\cite{MayTeePACC11} yielding global asymptotic stability. Another interesting feature is that the construction of the expelling configuration error functions are simpler as it is constructed on the unit-sphere. 

In~\cite{MayTeeA13}, a synergistic family of potential functions is constructed for the reduced attitude tracking of spherical directions through a parameterized deffeomorphism of the nominal configuration error function presented in this paper. 
Similarly in~\cite{MayTeePACC11}, an expelling configuration error function is constructed for the full attitude tracking by \textit{angular warping}, where the nominal configuration error function is composed with a diffeomorphism that represents stretched rotations. The resulting control system design involves nontrivial derivatives and it is relatively difficult to compute the required hysteresis gap $\delta$, that is required to implement the given hybrid controller.

In this paper, the construction of expelling configuration error function at \refeqn{PsiE1}, \refeqn{PsiE2} is intuitive and straightforward as they are based on the comparison between the modified desired directions and the actual directions, rather than composing the nominal configuration error function with a diffeomorphism as~\cite{MayTeePACC11,MayTeeA13}. As a result, a range of the hysteresis gap to guarantee stability is explicitly given by \refeqn{deltaSO}, which can be easily checked by given controller gains.  In short, the presented control system provide a stronger global exponential stability, with simpler controller design procedure of choosing a hystereses gap $\delta$. For example, at \refeqn{deltaSO}, the upper bound of $\delta$ is maximized along the line of $2\alpha-|\beta|-3=0$ to yield $0<\delta<2-\alpha$. If $\alpha=1.6$, then we can simply choose $\beta=0.2$ to obtain $0<\delta<0.4 \min\{k_1,k_2\}$.


\section{Robust Attitude Tracking with Disturbance}\label{sec:RAT}

In this section, we consider a case where there exists unknown, but fixed disturbance $\Delta$. Both of smooth and hybrid attitude control systems are constructed to achieve exponential stability in the presence of the disturbance via an integral control. 

\subsection{Almost Global Robust Attitude Tracking}\label{sec:RAG}

First, a smooth attitude control scheme is proposed in this subsection. It is based on constructing an estimate of the disturbance, denoted by $\bar\Delta\in\Re^3$ according to the following differential equation,
\begin{align}
\dot{\bar\Delta} &= \frac{k_\Delta}{2} \braces{e_\Omega + \parenth{c+\frac{1}{k_\Omega}}e_r},\label{eqn:bD_dot}
\end{align}
for positive constants $c,k_\Delta$. This is designed to achieve exponential convergence of the estimation error defined as $e_\Delta = \Delta-\bar\Delta\in\Re^3$, as well as the attitude tracking errors. 



\begin{prop}\label{prop:RAGAS}
Consider the dynamic system \refeqn{Wdot}, \refeqn{Rdot}. A desired trajectory is given by \refeqn{rddot}. For $k_1,k_2,k_\Omega,k_\Delta, c >0$ with $k_1\neq k_2$, a control input is chosen as
\begin{align}
u & = -e_r -k_\Omega e_\Omega +\Omega\times J\Omega -J(\hat\Omega R^T \omega_d-R^T \dot\omega_d)\nonumber\\
&\quad -\bar\Delta,
\label{eqn:uA}
\end{align}
where the estimate $\bar\Delta$ is constructed by \refeqn{bD_dot}. Then, there exists controller parameters such that the following properties hold:
\renewcommand{\labelenumi}{(\roman{enumi})}
\begin{enumerate}
\item There are four equilibrium configurations for $(R,\omega)$, given by the property (i) of Proposition \ref{prop:AGASSO}.
\item The desired equilibrium $(R_d,\omega_d)$ with $\bar\Delta=\Delta$ is almost globally asymptotically stable, and locally exponentially stable.
\item The three undesired equilibria are unstable.
\end{enumerate}
\end{prop}

\begin{proof}
From \refeqn{eW} and \refeqn{uA}, the error dynamics for $e_\Omega$ is given by
\begin{align}
J\dot e_\Omega & = -e_r -k_\Omega e_\Omega + e_\Delta.\label{eqn:JeWdot1}
\end{align}
Equilibria of the controlled system corresponds to the configurations where $e_r=0$, $e_\Omega=0$, and $e_\Delta=0$. From the proof of Proposition \ref{prop:AGASSO}, this shows (i).

Define an augmented angular velocity error vector $\bar e_\Omega\in\Re^3$ as
\begin{align}
\bar e_\Omega = e_\Omega -\frac{1}{2k_\Omega}e_\Delta,\label{eqn:beW}
\end{align}
which is introduced to show exponential convergence for all of the tracking errors and the estimation error. From \refeqn{JeWdot1}, \refeqn{beW}, and using the fact that $\dot e_\Delta=-\dot{\bar\Delta}$, the time-derivative of the augmented angular velocity error is given by
\begin{align}
J\dot{\bar e}_\Omega 
& = -e_r -k_\Omega \bar e_\Omega + \frac{1}{2} e_\Delta
+\frac{1}{2k_\Omega}J\dot{\bar \Delta}.\label{eqn:JbeWdot1}
\end{align}
Similarly, we have $\dot\Psi = e_r \cdot e_\Omega = e_r \cdot \bar e_\Omega + \frac{1}{2k_\Omega} e_r\cdot e_\Delta$. 

 
Define a Lyapunov function:
\begin{align*}
\bar{\mathcal{V}}= \frac{1}{2} \bar e_\Omega\cdot J \bar e_\Omega + \Psi + cJ\bar e_\Omega\cdot e_r + \frac{1}{2k_\Delta} e_\Delta\cdot e_\Delta.
\end{align*}
From \refeqn{PsibSO}, it is bounded by
\begin{align}
\bar z^T \bar M_1 \bar z 
\leq \bar{\mathcal{V}}
\leq \bar z^T \bar M_2 \bar z ,\label{eqn:VbI}
\end{align}
where $\bar z=[\|e_r\|,\|e_\Omega\|,\|e_\Delta\|]^T\in\Re^3$ and the matrices $\bar M_1,\bar M_2\in\Re^{3\times 3}$ are given by $\bar M_1 = \mathrm{diag}(M_1,\frac{1}{2k_\Delta})$ and $\bar M_2=\mathrm{diag}(M_2,\frac{1}{2 k_\Delta})$. The submatrices $M_1$ and $M_2$ are given at \refeqn{M1} and \refeqn{M2}. 
From \refeqn{JbeWdot1}, we have 
\begin{align*}
\dot{\bar{\mathcal{V}}} 
& = (\bar e_\Omega + c e_r)\cdot\{-e_r -k_\Omega \bar e_\Omega + \frac{1}{2} e_\Delta
+\frac{1}{2k_\Omega}J\dot{\bar \Delta}\}\\
&\quad + e_r \cdot \bar e_\Omega + \frac{1}{2k_\Omega} e_r\cdot e_\Delta
+ c \dot e_r \cdot J\bar e_\Omega 
- \frac{1}{k_\Delta} e_\Delta \cdot \dot{\bar \Delta}.
\end{align*}
The above expression is simplified as follows. First, the terms that are explicitly linear with respect to $e_\Delta$ can be rearranged by \refeqn{bD_dot} as
\begin{align*}
e_\Delta & \cdot \{\frac{1}{2}(\bar e_\Omega + c e_r) +\frac{1}{2k_\Omega} e_r - \frac{1}{k_\Delta} \dot{\bar\Delta} \}
 = e_\Delta \cdot \frac{1}{2} (\bar e_\Omega - e_\Omega)\\
& = -\frac{1}{4k_\Omega} \|e_\Delta\|^2.
\end{align*}
Second, from the property (ii) of Proposition \ref{prop:errSO}, we have 
\begin{align*}
 cJ\bar e_\Omega \cdot \dot e_r & \leq c\lambda_M(k_{1}+k_2) \|\bar e_\Omega\|\|e_\Omega\|\\
& \leq c\lambda_M (k_{1}+k_2)\parenth{\|\bar e_\Omega\|^2 + \frac{1}{2k_\Omega}\|\bar e_\Omega\|\|e_\Delta\|}.
\end{align*}
Next, from \refeqn{bD_dot} and \refeqn{beW},
\begin{align*}
\frac{1}{2k_\Omega}(\bar e_\Omega & + ce_r)\cdot J\dot{\bar \Delta} \\
&=(\bar e_\Omega + ce_r)\cdot \frac{k_\Delta J}{4k_\Omega} \braces{\bar e_\Omega +\frac{1}{2k_\Omega}e_\Delta+ \parenth{c+\frac{1}{k_\Omega}}e_r}\\
&\leq \frac{k_\Delta \lambda_M}{4k_\Omega} \{\|\bar e_\Omega\|^2 + \parenth{c^2+\frac{c}{k_\Omega}}\|e_r\|^2\\
&\quad +\parenth{2c+\frac{1}{k_\Omega}}\|\bar e_\Omega\|\|e_r\|
 + \frac{1}{2k_\Omega}(\|\bar e_\Omega\|+c\|e_r\|)\|e_\Delta\|\}.
\end{align*}
Using these, an upper bound of the time-derivative of the Lyapunov function can be written as
\begin{align*}
\dot{\bar{\mathcal{V}}} & \leq -\bar z^T \bar M_3 \bar z,
\end{align*}
where the matrix $\bar M_{3}\in\Re^{3\times 3}$ is given at \refeqn{bM3}, and the unspecified parts of \refeqn{bM3} is chosen such that $\bar M_3=\bar M_3^T$. 
\begin{figure*}
\begin{align}
\bar M_3 = \begin{bmatrix}
c(1-\frac{k_\Delta\lambda_M}{4k_\Omega}(c+\frac{1}{k_\Omega})) 
& \frac{ck_\Omega}{2}-\frac{k_\Delta\lambda_M}{8k_\Omega}(2c+\frac{1}{k_\Omega}) & 
-\frac{ck_\Delta\lambda_M}{16k_\Omega^2}\\
 \cdot & k_\Omega-\lambda_M(ck+\frac{k_\Delta}{4k_\Omega}) 
   & -\frac{c\lambda_M k}{4k_\Omega}-\frac{k_\Delta\lambda_M}{16k_\Omega^2}\\
\cdot &  \cdot       & \frac{1}{4k_\Omega}
\end{bmatrix}.\label{eqn:bM3}
\end{align}
\end{figure*}
There exist the values of controller parameters such that the matrix $\bar M_3$ becomes positive define. For example, when $c=\epsilon$, $\lambda_M k_\Delta=\epsilon$, $\lambda_M (k_1+k_2) = \epsilon$ and $k_\Omega=\frac{1}{\epsilon}$ for a  constant $\epsilon$, we can show that $\bar M_3$ is positive definite when $0<\epsilon <0.85$ from the Matlab symbolic computational tool. This implies that the desired equilibrium is asymptotically stable. The instability of the undesired equilibria can be shown by following the same approach given at the proof of Proposition \ref{prop:AGASSO}. These show almost global asymptotic stability. 

For exponential stability, the upper bound of \refeqn{VbI} should be satisfied, or $\Psi< h_1$ from Proposition \ref{prop:errSO}. Unlike the proof of Proposition \ref{prop:AGASSO}, we do not have a sufficient condition on the initial conditions for the bound. As such, we can only guarantee local exponential stability.
\end{proof}

The estimation law presented at \refeqn{bD_dot} can be interpreted as an integral control. The first term $e_\Omega$ at the right hand side of \refeqn{bD_dot} has an effect of increasing the proportional gain of the control system, as the time-derivative of the error vector, namely $\dot e_r$ is linear with respect to the angular velocity error $e_\Omega$. Effectively, the proportional gain of the control input is given by $k_1,k_2$ multiplied by $1+\frac{k_\Delta}{2}$, and the integral gain of the control input is given by $\frac{k_\Delta}{2}(c+\frac{1}{k_\Omega})$. 

Nonlinear PID-like controllers have been developed for attitude stabilization in terms of modified Rodriguez parameters~\cite{SubJAS04} and quaternions~\cite{SubAkeJGCD04}, and for attitude tracking in terms of Euler-angles~\cite{ShoJuaPACC02}. The proposed control system is developed on $\SO$, therefore it avoids singularities of Euler-angles and Rodriguez parameters, as well as unwinding of quaternions. It also provides almost global asymptotic stability for attitude \textit{tracking} problems with fixed uncertainties.

One of the unique feature of the presented control system is that it guarantees exponential convergence of the estimation error $e_\Delta$, as well as the tracking errors $e_r$, $e_\Omega$. This is in contrast to most of other  indirect adaptive control approaches where there is no guarantee on the convergence rate of the parameter estimation error. 



\subsection{Global Robust Attitude Tracking}\label{sec:RGES}

The preceding attitude control system is further developed into a hybrid control system to achieve global exponential stability in the presence of the disturbance. The estimation of the disturbance is redefined in terms of the hybrid error vector given at \refeqn{eH} as
\begin{align}
\dot{\bar\Delta}_H &= \frac{k_\Delta}{2} \braces{e_\Omega + \parenth{c+\frac{1}{k_\Omega}}e_H},\label{eqn:bDH_dot}
\end{align}
where $\bar\Delta_H\in\Re^3$ denotes an estimate of the disturbance. The jump set and the flow set are revised as
\begin{align}
\bar{\mathcal{D}} & = \{ (R,\Omega,\mb) : \Psi_\mb - \rho \geq \delta\text{ and } E(e_\Omega,\bar\Delta_H) \leq \frac{\delta}{4}\},\label{eqn:DI}\\
\bar{\mathcal{C}} & = \{ (R,\Omega,\mb) : \Psi_\mb - \rho \leq \delta\text{ or } E(e_\Omega,\bar\Delta_H) \geq \frac{\delta}{4}\},\label{eqn:CI}
\end{align}
where $E(e_\Omega,\bar\Delta_H)\in\Re$ is a scalar function of $e_\Omega,\bar\Delta_H$ defined as
\begin{align}
E(e_\Omega,\bar\Delta_H)=2ck \|e_\Omega\| + \frac{ck}{k_\Omega}\|\bar\Delta_H\|.
\label{eqn:E}
\end{align}

The control input is chosen as
\begin{align}
u & = -e_H -k_\Omega e_\Omega +\Omega\times J\Omega -J(\hat\Omega R^T \omega_d-R^T \dot\omega_d)\nonumber\\
&\quad -\bar\Delta_H.\label{eqn:uRGES}
\end{align}
The other parts of the hybrid control system, such as the jump map are identical to Section \ref{sec:GES}.

\begin{prop}\label{prop:RGES}
Consider a hybrid control system defined by \refeqn{PsiI}-\refeqn{GSO}, \refeqn{eH}-\refeqn{eH2}, and \refeqn{bDH_dot}-\refeqn{CI}. For given constants $k_1,k_2,\alpha,\beta$ satisfying $k_1,k_2>0$, $k_1\neq k_2$, $1<\alpha<2$ and $|\beta|< \alpha-1$, choose the hysteresis gap $\delta$ such that \refeqn{deltaSO} is satisfied. Assume that the bound of the disturbance given at \refeqn{Bd} satisfies $B_\Delta < \frac{\delta k_\Omega}{4ck}$. Then, the desired equilibrium $(R_d,\omega_d)$ is globally exponentially stable with respect to $\bar z=[\|e_H\|,\, \|e_\Omega\|,\, \|e_\Delta\|]\in\Re^3$.
\end{prop}

\begin{proof}
As shown at the proof of Proposition \ref{prop:GES}, all of the undesired critical points of the configuration error functions lie in the jump set $\bar{\mathcal{D}}$, and Proposition \ref{prop:errSO} with \refeqn{PsiBgam} is satisfied in the flow set $\bar{\mathcal{C}}$. Define a Lyapunov function as 
\begin{align*}
\bar{\mathcal{V}}_\mb = \frac{1}{2} \bar e_\Omega \cdot J \bar e_\Omega + \Psi_\mb
 + c \bar e_\Omega\cdot e_H + \frac{1}{2k_\Delta}e_\Delta\cdot e_\Delta,
\end{align*}
where $e_\Delta = \Delta -\bar\Delta_H$. From the proof of Proposition \ref{prop:RAGAS}, $\dot{\bar{\mathcal{V}}}_\mb$ is negative definite with respect to a negative quadratic function of $e_H$, $\bar e_\Omega$, and $e_\Delta$, and therefore all of the error variables exponentially decrease in the flow set $\bar{\mathcal{C}}$.

Since $\bar e_\Omega, e_\Delta$ are not changed over any jump, the change of the Lyapunov function over the jump from a mode $\mb\in\mathcal{M}$ is
\begin{align*}
\bar{ \mathcal{V}}_\mathcal{G}-\bar{\mathcal{V}}_\mb&=\rho-\Psi_{\mb} + c\bar e_\Omega\cdot e_H\big|^{\mathcal{G}}_\mb\\
& \leq -\delta + 2ck\braces{\|e_\Omega\|+\frac{1}{2k_\Omega}(B_\Delta + \|\bar\Delta_H\|)}\end{align*}
where we use the fact that $\|e_H\|\leq k$. From the definition of the jump set given at \refeqn{DI}, and using the assumption implying $\frac{ck}{k_\Omega}B_\Delta < \frac{\delta}{4}$, we have $\bar{\mathcal{V}}_\mathcal{G}-\bar{\mathcal{V}}_\mb\leq -\frac{\delta}{2} < 0$, which implies that the Lyapunov function strictly decreases in the jump set $\bar{\mathcal{D}}$. Therefore, the desired equilibrium is global exponentially stable.
\end{proof}

Global asymptotic stability is achieved for an attitude control system with an integral term in terms of quaternions, based on LaSalle's principle~\cite{SuCaiJGCD11}. The proposed control system guarantees a stronger exponential stability of all of the tracking errors and the estimation errors in the presence of the disturbance.


\section{Numerical Examples}\label{sec:NS}

Consider a rigid body whose inertia matrix is given by $J = 0.1\times \mathrm{diag}[3,2,1]\,\mathrm{kgm^2}$. The desired attitude command is specified as $R_d(t) =\exp(\psi(t)\hat e_3)\exp(\theta(t)\hat e_2)\exp(\phi(t)\hat e_1)$ in terms of $3\text{-}2\text{-}1$ Euler-angles, where $\phi(t)=\sin 0.5t$, $\theta(t)=0.1(-1+t)$, $\psi(t)=1-\cos t$. The controller parameters are chosen as $b_1=[1,0,0]^T$, $b_2=[0,1,0]^T$, $\alpha=1.9$, $\beta= 0.8$, $\delta=0.39$, $k_1=4$, $k_2=4.1$, $k_\Omega=2.8$, $k_I=2$, and $c=0.1$. The following three cases are considered.

\begin{figure}
\centerline{
	\subfigure[Attitude tracking error $\|R-R_d\|$]{
		\includegraphics[width=0.48\columnwidth]{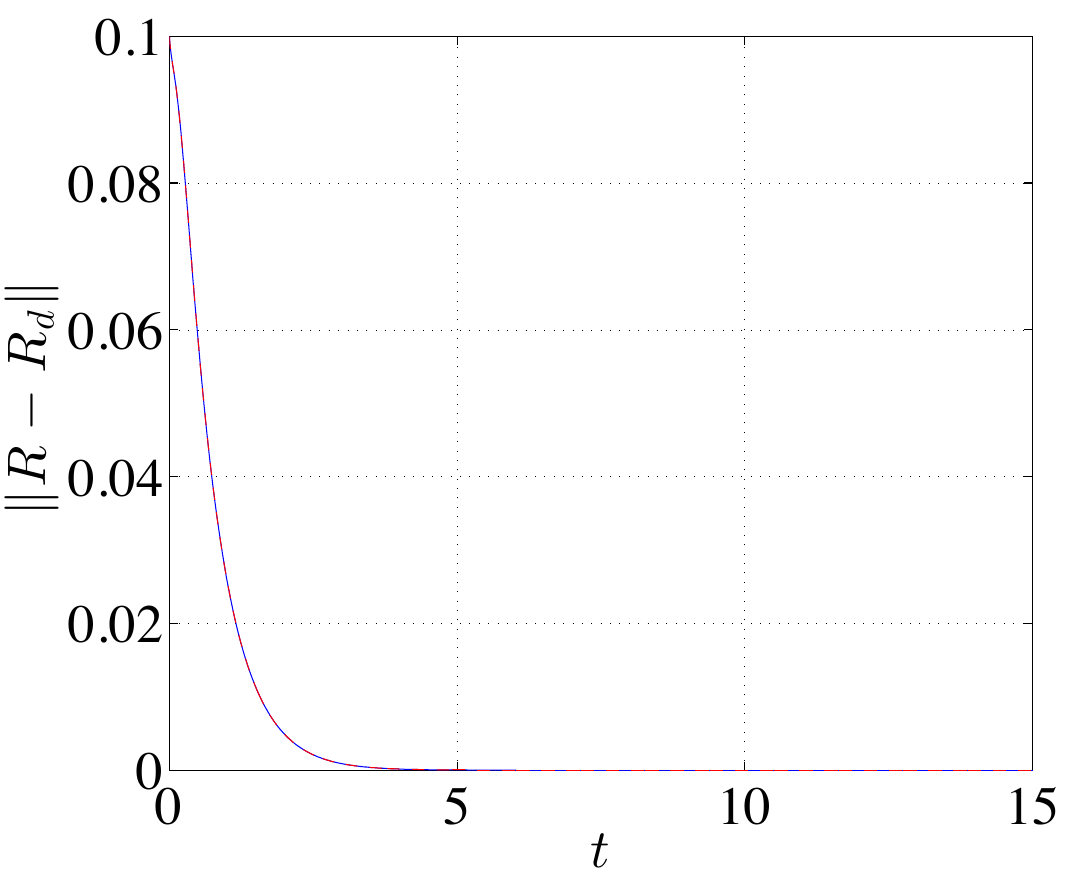}}
	\subfigure[Attitude error vector $e_H$]{
		\includegraphics[width=0.52\columnwidth]{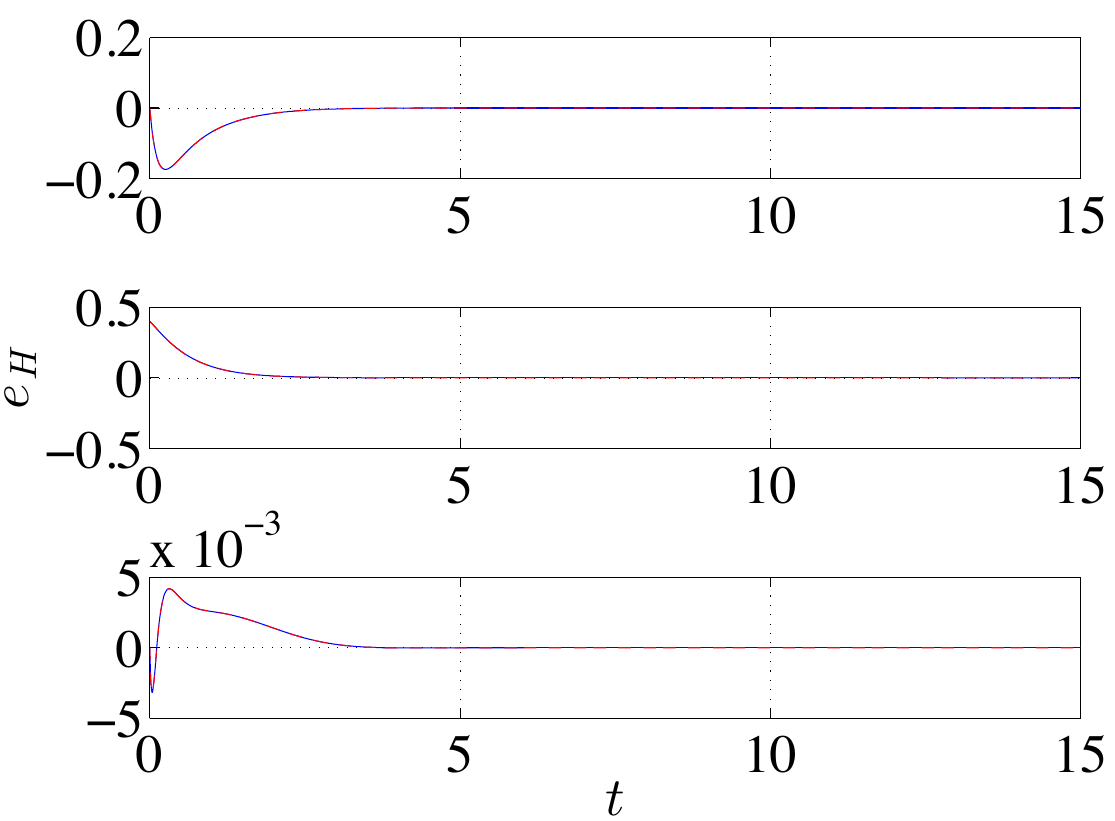}}
}
\centerline{
	\subfigure[Angular velocity error $e_\Omega$]{
		\includegraphics[width=0.49\columnwidth]{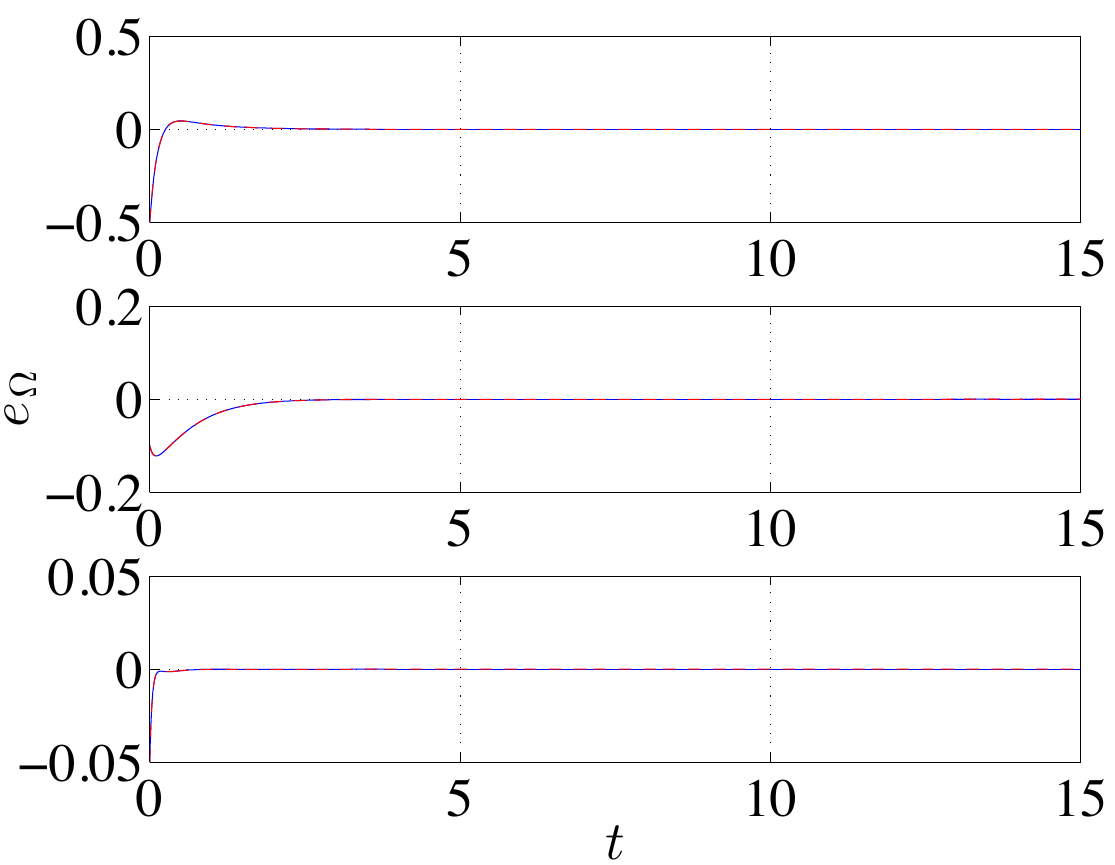}}
	\subfigure[Control input $u$]{
		\includegraphics[width=0.48\columnwidth]{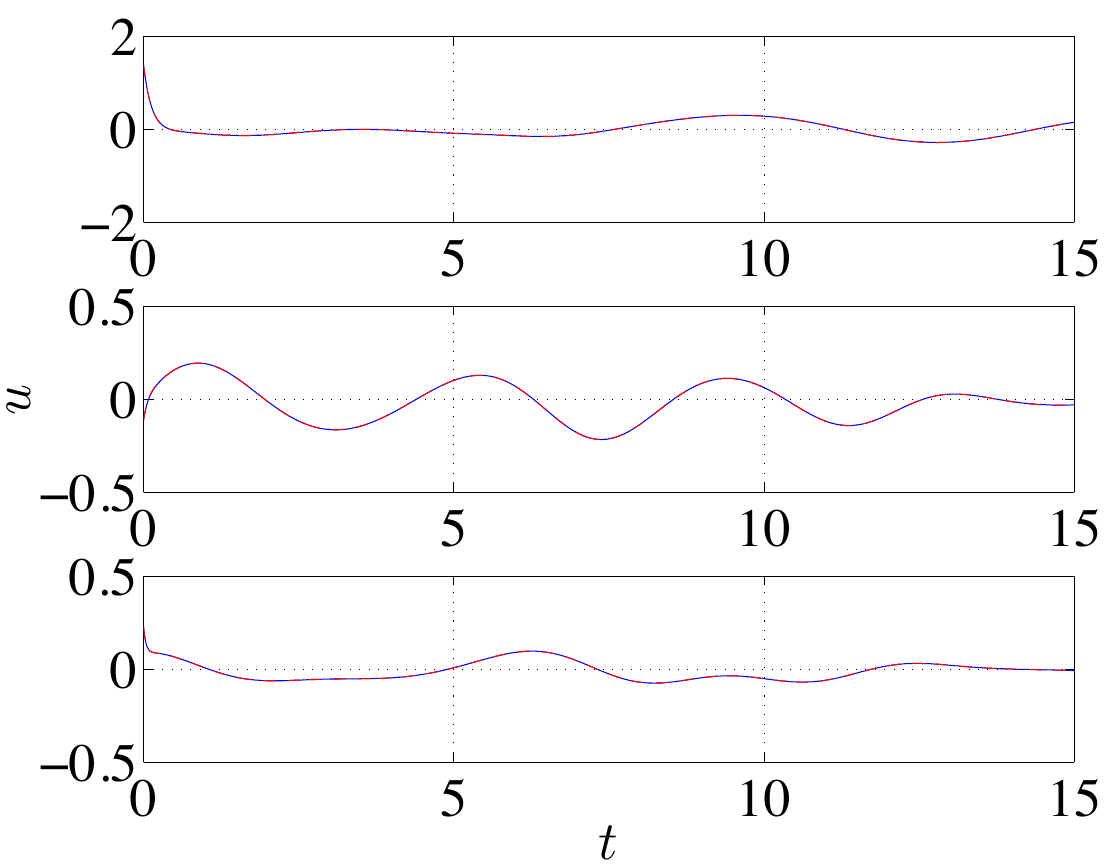}}
}
\caption{Case (i): Small initial attitude error without disturbances (blue,solid:smooth controller, red,dashed:hybrid controller)}\label{fig:1}
\end{figure}

\textit{Case (i):\quad} It is assumed that there is no disturbance, i.e., $\Delta=0$, and the initial conditions are chosen as $R(0)=I$ and $\Omega(0)=0$. This corresponds to a small initial attitude error, where $\Psi(0)=0.02$. The simulation results for the smooth control system and the hybrid control system without the integral control term term, developed at Propositions \ref{prop:AGASSO} and \ref{prop:GES} respectively, are illustrated at Figure \ref{fig:1}. They exhibit good tracking performances.  As the initial attitude error is small, no jump occurs at the hybrid control system, and the corresponding responses of the hybrid control system are identical to the smooth control system.

\textit{Case (ii):\quad} The second case is same as Case (i), except the initial condition chosen as $R(0) = \exp(0.9999\pi (r_{1_d}\times r_{2_d})^\wedge) R_d(0)$, $\Omega(0)=R(0)^T \omega_d(0)$, which is close to one of the undesired equilibrium. In this case, there is noticeable difference between the smooth controller and the hybrid controller, as illustrated at Figure \ref{fig:2}. For the smooth controller, the attitude tracking error does not change until after $t=12$ seconds. This is because the attitude error vector $e_r$ is close to zero initially, even though the initial attitude error is almost $180^\circ$. For the proposed hybrid control system, there is a mode switching from $\mb=\RII$ to $\mb=\RI$ at $t=3.74$ seconds, and the corresponding convergence rate is significantly faster.

\textit{Case (III):\quad} The initial condition is identical to Case (ii), representing a large initial attitude error. In this case, a fixed disturbance of $\Delta=[-0.4,0.8,0.4]^T$ is included. Figure \ref{fig:3} shows numerical results for the hybrid control system presented at Proposition \ref{prop:GES}, and the hybrid control system with an integral term presented at Proposition \ref{prop:RGES} with the initial estimate $\bar\Delta_H(0)=0$. The given fixed disturbance causes steady-state tracking errors for the hybrid control system developed at Proposition \ref{prop:GES}, but those errors are completely eliminated by the integral term of the hybrid control system developed at Proposition \ref{prop:RGES}. It also exhibits good convergence properties for the given large initial attitude error, which are comparable to the hybrid control system without disturbances illustrated at Figure \ref{fig:2}.

\section{Conclusions}

Four types of attitude tracking control systems are developed in this paper. A smooth attitude control system is presented for almost semi-global exponential stability, and a new form of synergistic attitude error functions are introduced for global exponential stability. They are further extended to obtain robustness with respect to a fixed disturbance. The main contribution is achieving global exponential stability on the special orthogonal group for all of the tracking error variables and the estimation errors in the presence of uncertainties. Future directions include generalizing the presented results into global adaptive attitude controls by incorporating parametric uncertainties in the attitude dynamics.

\appendix

\begin{figure}
\centerline{
	\subfigure[Attitude tracking error $\|R-R_d\|$]{
		\includegraphics[width=0.48\columnwidth]{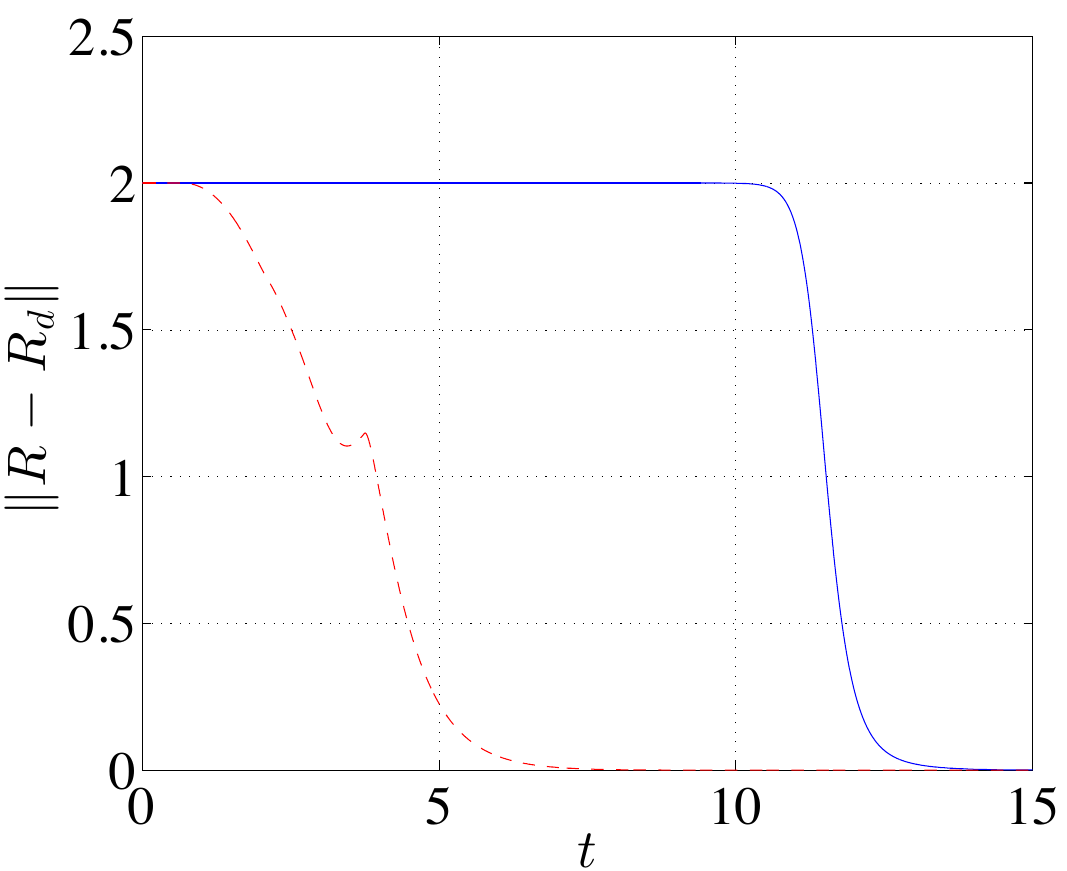}}
	\subfigure[Attitude error vector $e_H$]{
		\includegraphics[width=0.48\columnwidth]{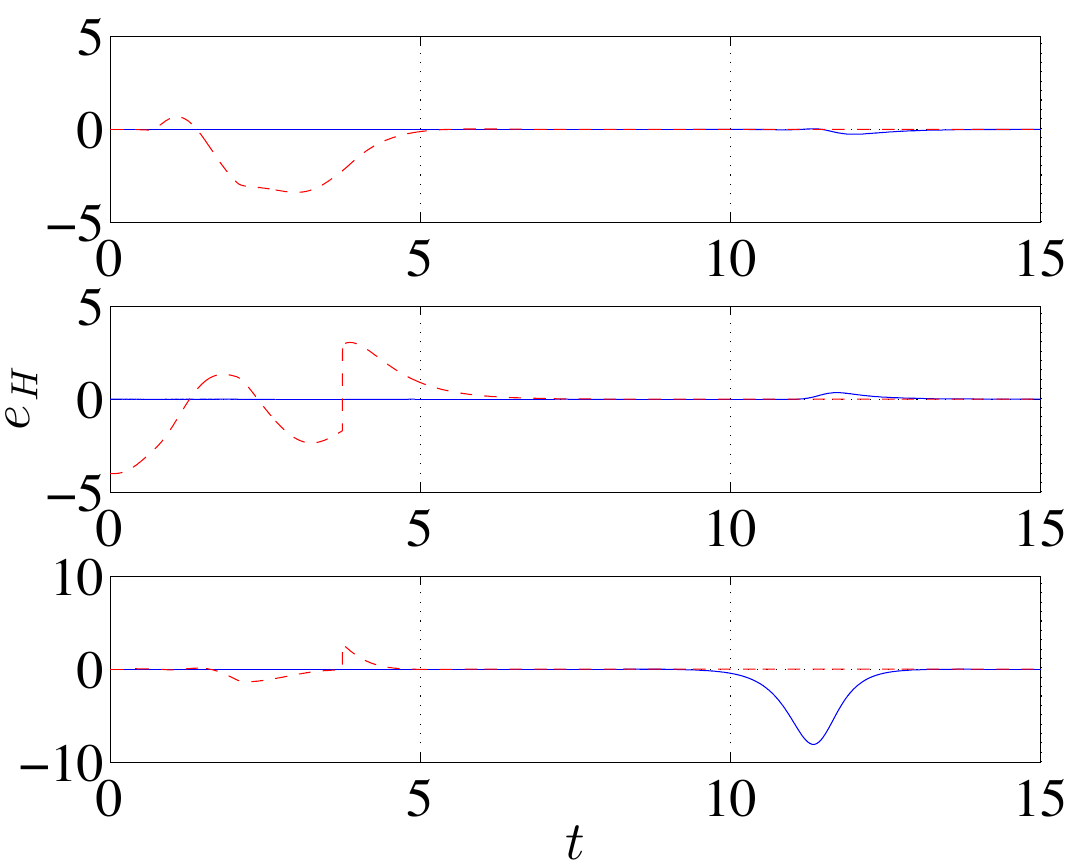}}
}
\centerline{
	\subfigure[Angular velocity error $e_\Omega$]{
		\includegraphics[width=0.48\columnwidth]{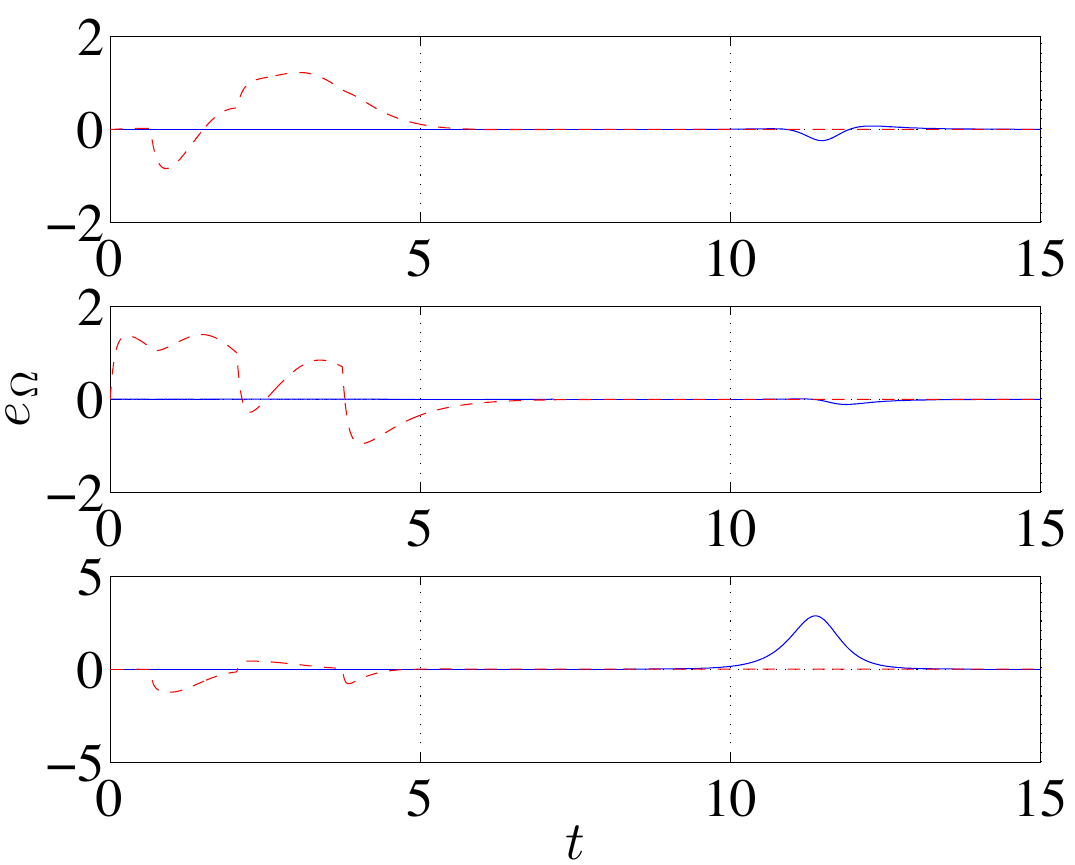}}
	\subfigure[Control input $u$]{
		\includegraphics[width=0.48\columnwidth]{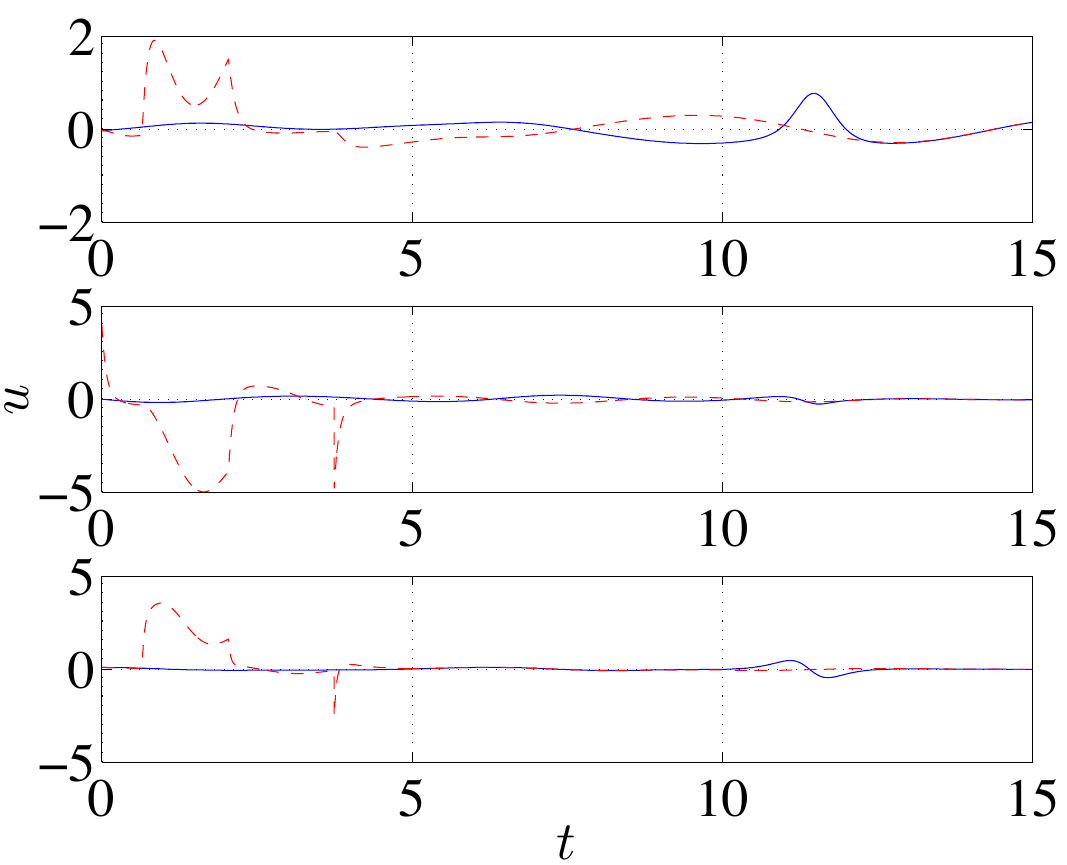}}
}
\caption{Case (ii): Large initial attitude error without disturbances (blue,solid:smooth controller, red,dashed:hybrid controller)}\label{fig:2}
\end{figure}

\begin{figure}
\centerline{
	\subfigure[Attitude tracking error $\|R-R_d\|$]{
		\includegraphics[width=0.48\columnwidth]{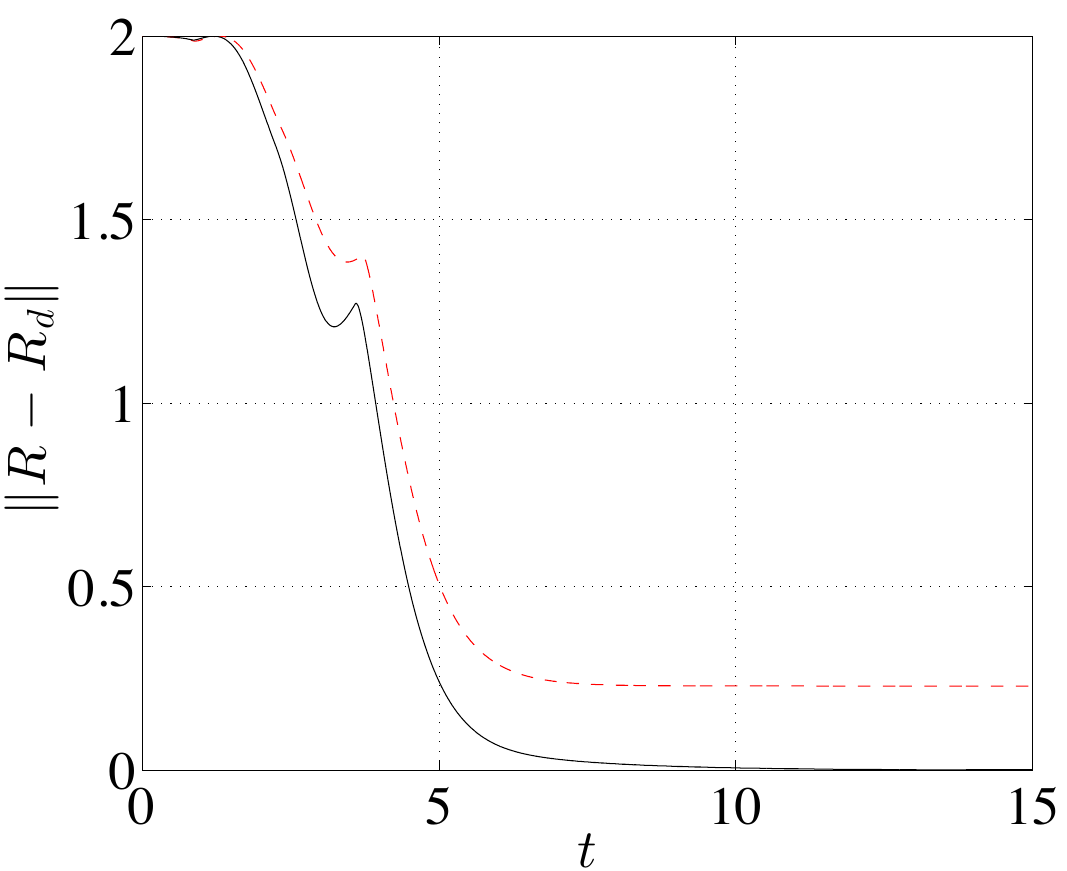}}
	\subfigure[Attitude error vector $e_H$]{
		\includegraphics[width=0.48\columnwidth]{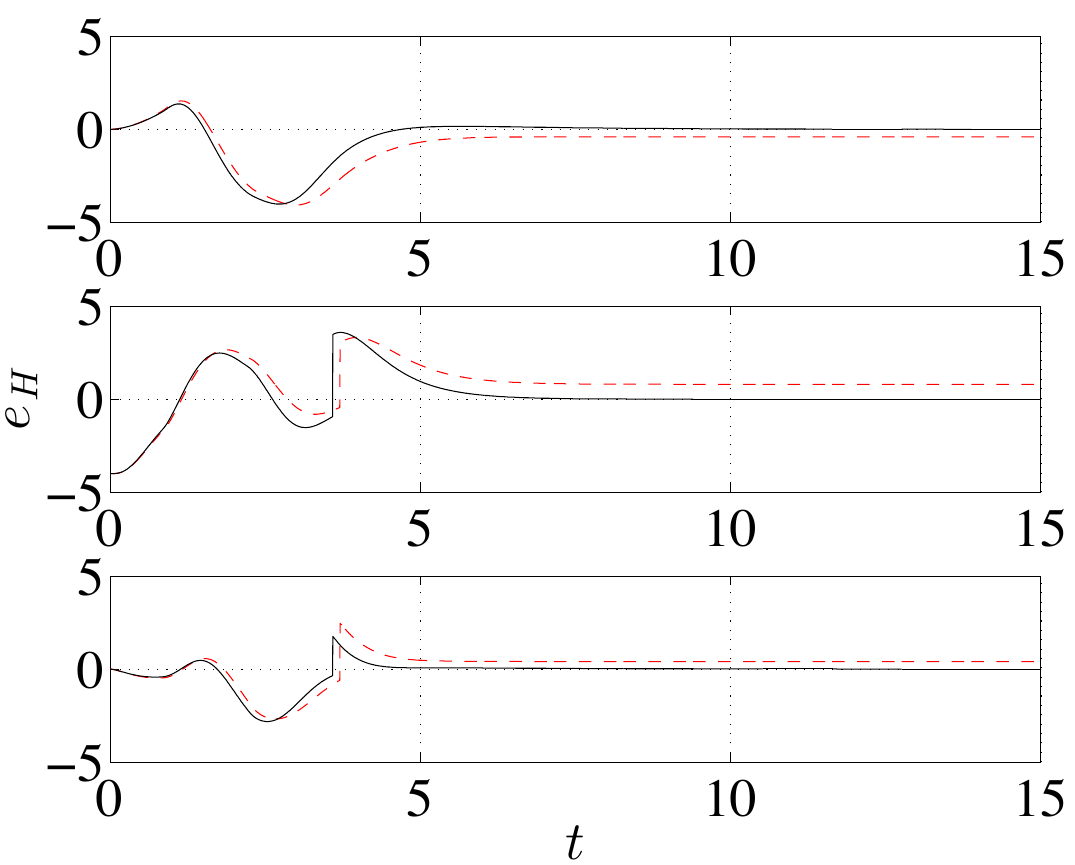}}
}
\centerline{
	\subfigure[Angular velocity error $e_\Omega$]{
		\includegraphics[width=0.48\columnwidth]{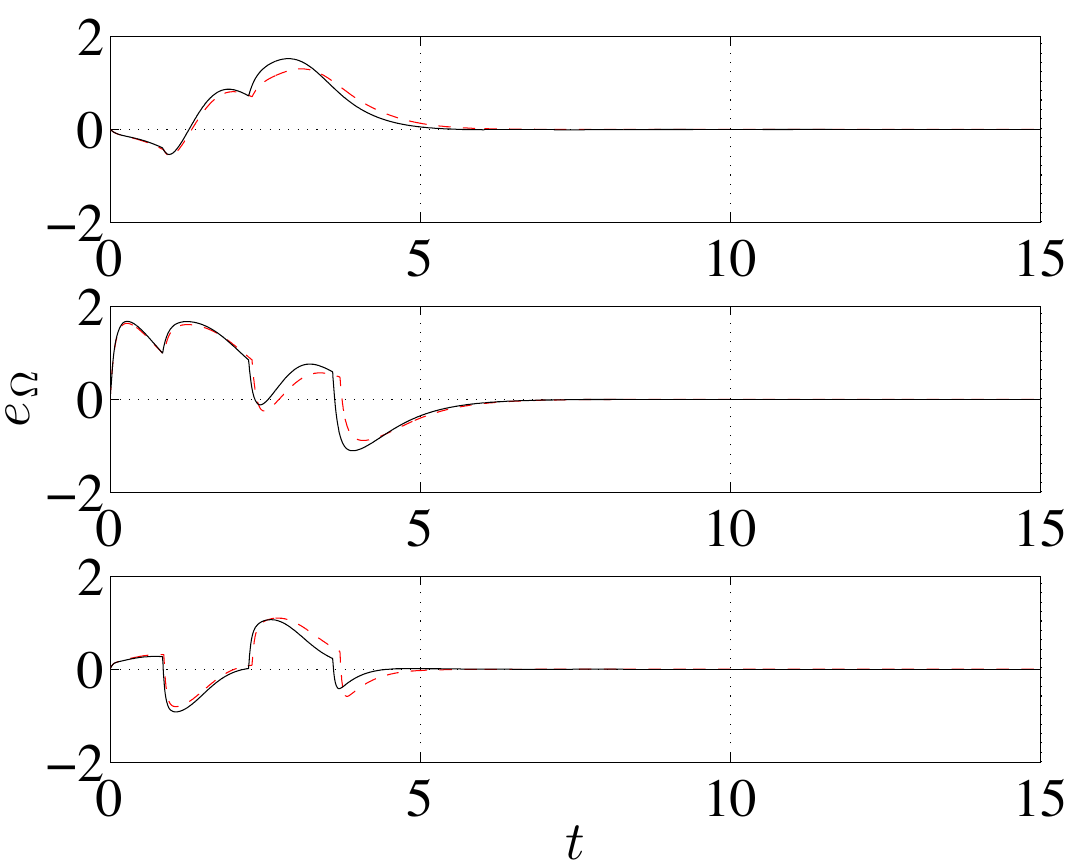}}
	\subfigure[Control input $u$]{
		\includegraphics[width=0.48\columnwidth]{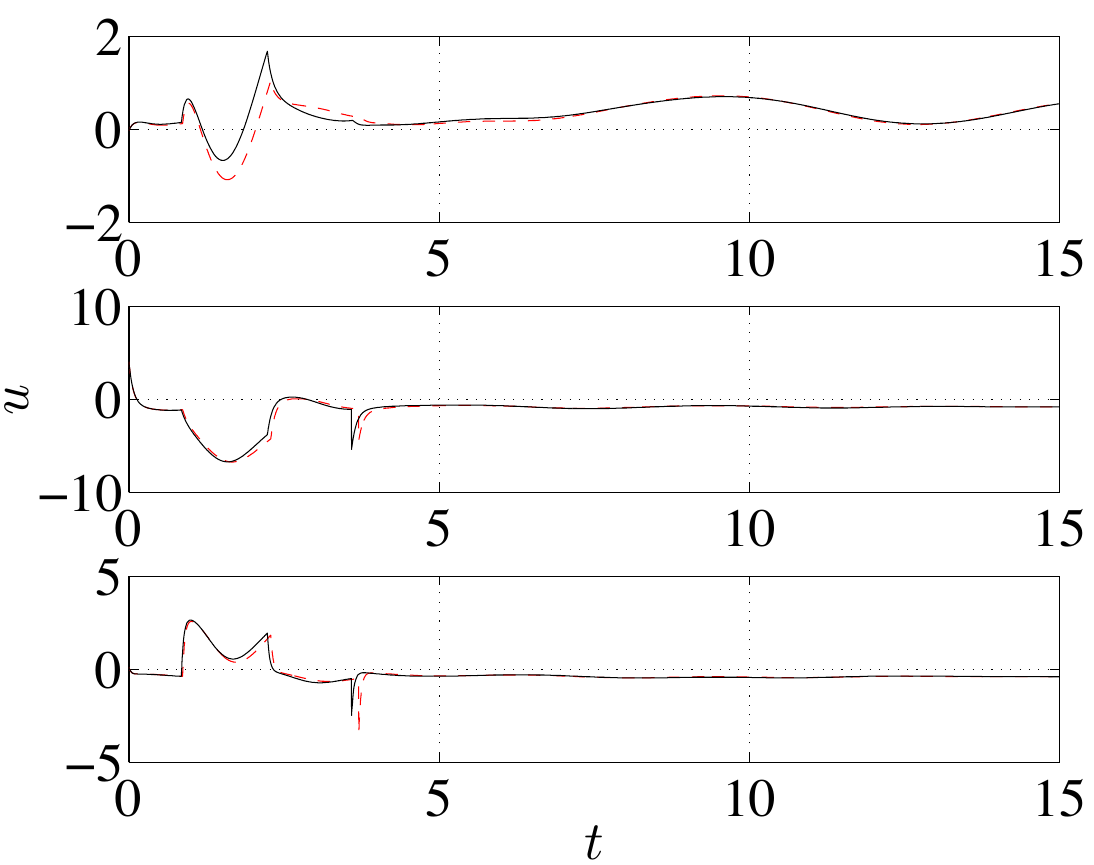}}
}
\centerline{
	\subfigure[Estimation error $e_\Delta$]{
		\includegraphics[width=0.48\columnwidth]{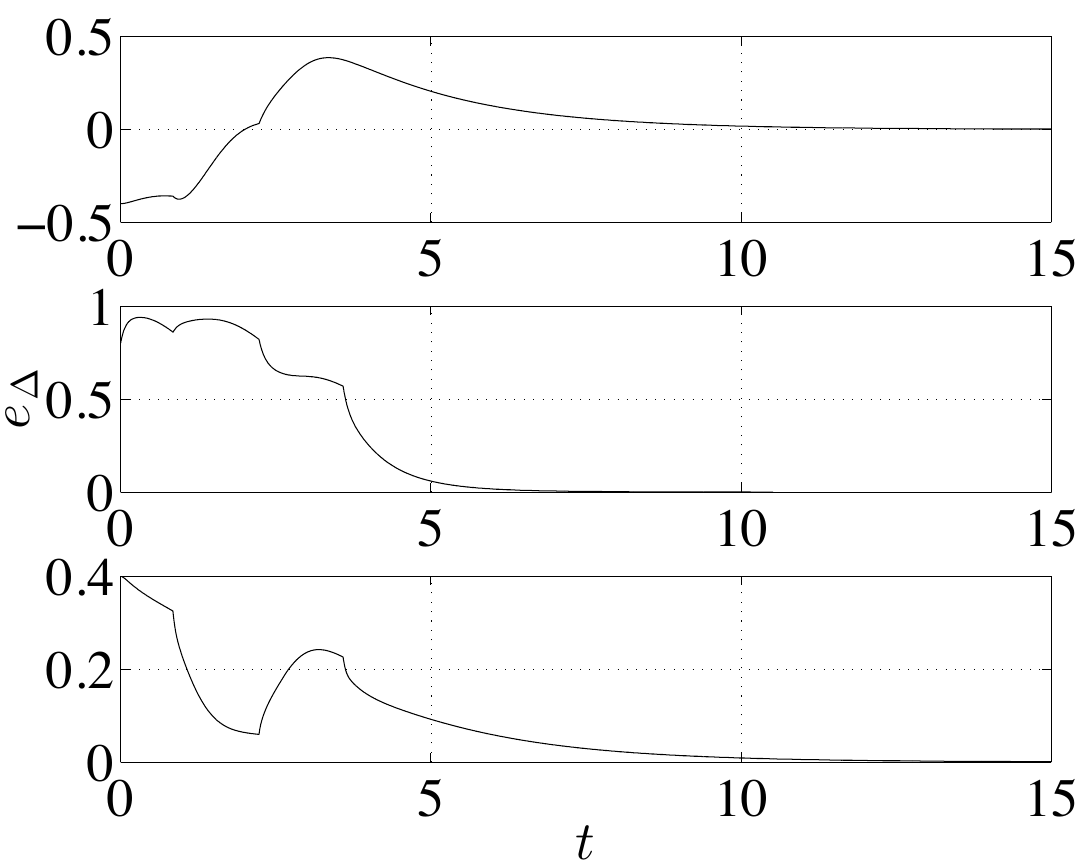}}
}
\caption{Case (iii): Large initial attitude error with disturbances (red,dashed:hybrid controller, black,solid:hybrid controller with an integral term)}\label{fig:3}
\end{figure}

\subsection{Proof of Proposition \ref{prop:errSO}}\label{sec:errSO}

From \refeqn{Rdot} and \refeqn{rddot}, and using \refeqn{STP} and \refeqn{RxR},
\begin{align*}
\dot \Psi_i & = -R\hat\Omega b_i\cdot r_{i_d} - Rb_i\cdot (\omega_d\times r_{i_d})\\
& = \Omega_i \cdot (R^T r_{i_d} \times b_i)- b_i\cdot (R^T\omega_d\times R^Tr_{i_d}).
\end{align*}
Substituting $\Omega=e_\Omega + R^T\omega_d$ into this, we obtain (i).

Using \refeqn{hatxy} and \refeqn{rdi}, $\hat e_{r_i}$ can be written as
\begin{align}
\hat e_{r_i} = b_i r_{i_d}^T R - R^T r_{i_d} b_i^T = b_i b_i^T R_d^T R - R^T R_d b_{i} b_i^T.\label{eqn:eri2}
\end{align}
%
Using \refeqn{RxR}, the time-derivative of $R_d^T R$ is given by
\begin{align}
\frac{d}{dt} (R_d^T R) = -R_d^T\hat\omega_d R + R_d^T R\hat\Omega = R_d^T R \hat e_\Omega.\label{eqn:Redot}
\end{align}
Therefore, we have
\begin{align*}
\hat{\dot e}_{r_i} & = b_i b_i^T R_d^T R\hat e_\Omega +\hat e_\Omega R^T R_d b_{i} b_i^T,\\
& = (\hat b_i \hat e_\Omega R^T R_d b_{i})^\wedge,
\end{align*}
where we used \refeqn{hatxy}. This shows (ii).

Since $x^Ty=\tr{xy^T}$ for any $x,y\in\Re^3$, we have
\begin{align}
\Psi  
& = k_1+k_2-\tr{R(k_1b_1r_{1_d}^T +k_2b_2r_{2_d}^T)}\nonumber\\
& =k_1+k_2-\tr{R(k_1b_1b_1^T+k_2b_2b_2^T)R_d^T}\nonumber\\
& = \tr{G(I-U^T R_d^TRU)},\label{eqn:Psi00}
\end{align}
where $G=\mathrm{diag}[k_1,k_2,0]\in\Re^{3\times 3}$, and $U=[b_1,b_2,b_1\times b_2]\in\SO$. From \refeqn{eri}, the error vector $e_r$ can be rewritten as
\begin{align}
\hat e_r = UGU^T R_d^TR -R^T R_d UGU^T.\label{eqn:er0}
\end{align}

Next, we use the following properties given in~\cite{FerChaPICDC11}. For non-negative constants $f_1,f_2,f_3$, let $F=\text{diag}[f_1,f_2,f_3]\in\Re^{3\times 3}$, and let $P\in\SO$. Define
\begin{gather} 
\Phi=\frac{1}{2}\text{tr}[F(I-P)],\label{eqn:Phi}\\
e_P=\frac{1}{2}(FP-P^{\T}F)^\vee,\label{eqn:eP}
\end{gather}
Then, $\Phi$ is bounded by the square of the norm of $e_P$ as
\begin{gather}
\frac{h_1}{h_2+h_3}\|e_P\|^2\leq \Phi \leq\frac{h_1h_4}{h_5(h_1-\phi)}\|e_P\|^2, 
\label{eqn:PhiB}
\end{gather}
if $\Phi<\phi<h_1$ for a constant $\phi$, where $h_i$ are given by
\begin{align*} 
h_1 &= \text{min}\{f_1+f_2,~ f_2+f_3,~ f_3+f_1\}, \\
h_2 &= \text{max}\{(f_1-f_2)^2,~ (f_2-f_3)^2,~ (f_3-f_1)^2\}, \\
h_3 &= \text{max}\{(f_1+f_2)^2,~ (f_2+f_3)^2,~ (f_3+f_1)^2\}, \\
h_4 &= \text{max}\{f_1+f_2,~ f_2+f_3,~ f_3+f_1\}, \\
h_5 &= \text{min}\{(f_1+f_2)^2,~(f_2+f_3)^2,~(f_3+f_1)^2\}.
\end{align*}

Note that if we choose $F=2G$ and $P=U^TR_d^TRU$, then we have $\Psi=\Phi$. Substituting these into \refeqn{eP},
\begin{align*}
\hat e_P &= (GU^TR_d^TRU-  U^TR^TR_dUG)\\
&= U^T(UGU^TR_d^TR- R^TR_dUGU^T)U= U^T \hat e_r U
\end{align*}
from \refeqn{er0}. Therefore, $\|e_P\|=\|U e_r\| = \|e_r\|$. Substituting this into \refeqn{PhiB}, we obtain \refeqn{PsibSO}.

\bibliography{CDC12.1}
\bibliographystyle{IEEEtran}

\end{document}